\documentclass[final,leqno,onefignum,onetabnum]{siamltex1213}

\usepackage{amsmath,amssymb}
\usepackage{hyperref}
\hypersetup{linkcolor=blue}
\usepackage{float}

\def\R{{\Bbb R}}

\def\Om{\Omega}

\def\f{\frac}
\def\p{\partial}
\def\q{\quad}
\def\na{\nabla}

\def\n{\mathbf{n}}

\def\x{\mathbf{x}}
\def\y{\mathbf{y}}
\def\z{\mathbf{z}}
\def\A{{\mathbf A}}
\def\B{\mathbf{B}}

\def\bi{\begin{itemize}} \def\ei{\end{itemize}}
\def\be{\begin{eqnarray*}}
\def\ee{\end{eqnarray*}}
\def\eref#1{(\ref{#1})}
\def\diam{\mathop{\rm diam}\nolimits}
\def\etal{{\it et al}\;}

\def\0{{\mathbf 0}}
\newcommand{\beq}{\begin{equation}}
\newcommand{\eeq}{\end{equation}}

\title{ Electrical impedance tomography-based pressure-sensing using conductive membrane }

\author{Habib Ammari\footnotemark[2]\
\and Kyungkeun Kang\footnotemark[3]\ \footnotemark[5]
\and Kyounghun Lee\footnotemark[4]\ \footnotemark[5]
\and Jin Keun Seo\footnotemark[4]\ \footnotemark[5]
}

\begin{document}
\maketitle

\renewcommand{\thefootnote}{\fnsymbol{footnote}}

\footnotetext[2]{Department of Mathematics and Applications, Ecole Normale Sup\'erieure, 45 Rue d'Ulm, 75005 Paris, France ({\tt habib.ammari@ens.fr}).}
\footnotetext[3]{Department of Mathematics, Yonsei University 50 Yonsei-Ro, Seodaemun-Gu, Seoul 120-749, Korea ({\tt kkang@yonsei.ac.kr}). }
\footnotetext[4]{Department of Computational Science and Engineering, Yonsei University 50 Yonsei-Ro, Seodaemun-Gu, Seoul 120-749, Korea ({\tt imlkh@yonsei.ac.kr, seoj@yonsei.ac.kr}).  }
\footnotetext[2]{The first author was supported  by the ERC Advanced Grant Project MULTIMOD--267184.}
\footnotetext[5]{The second, third, and fourth authors were supported by the National Research Foundation of Korea (NRF) grant funded by the Korean government (MEST) (No. 2011-0028868, 2012R1A2A1A03670512).}

\renewcommand{\thefootnote}{\arbic{footnote}}

\slugger{siap}{xxxx}{xx}{x}{x--x}

\begin{abstract} This paper presents a mathematical framework for a flexible pressure-sensor model using electrical impedance tomography (EIT). When  pressure is applied to a conductive membrane patch with clamped  boundary, the pressure-induced surface deformation results in a change in the conductivity distribution. This change can be detected in the current-voltage data ({\it i.e.,} EIT data) measured on the boundary of the membrane patch. Hence, the corresponding inverse problem is to reconstruct the pressure distribution from the data.  The 2D apparent conductivity (in terms of EIT data) corresponding to the surface deformation is anisotropic. Thus, we consider a constrained inverse problem by restricting the coefficient tensor to the range of the map from pressure to 2D-apparent conductivity. This paper provides theoretical grounds for the mathematical model of the inverse problem. We develop a reconstruction algorithm based on a careful sensitivity analysis. We demonstrate the performance of the reconstruction algorithm through numerical simulations to validate its feasibility for future experimental studies.
\end{abstract}

\begin{keywords}
electrical impedance tomography, pressure sensing, conductive membrane, inverse problem, prescribed mean curvature equation.
\end{keywords}

\begin{AMS}
35R30, 35J25, 53A10
\end{AMS}


\pagestyle{myheadings}
\thispagestyle{plain}
\markboth{H. Ammari, K. Kang, K. Lee, and J. K. Seo}{Electrical impedance tomography-based pressure-sensing}

\section{Introduction}
There is a growing demand for cost-effective flexible pressure sensors. These devices have  wide   potential applicability, including in smart textiles \cite{Comert2013,Loruss2004,Loyola2010}, touch screens \cite{Han2014}, artificial skins \cite{Tawil2011}, and wearable health monitoring technologies \cite{Li2011,Merritt2009}. Electrical measurements have recently been used to measure the pressure-induced surface deformation of  conductive membranes. In particular, electrical impedance tomography (EIT) has been used to develop  flexible pressure sensors \cite{Peterson2011,YLeq2012,Yao2013}, because it allows the electromechanical behavior of an electrically conducting film to be monitored. When a pressure-sensitive conductive sheet is exposed to pressure, the deformation of the surface  alters the conductivity distribution, which can be detected by an EIT system. However, rigorous  studies employing mathematical modeling and reconstruction methods have not yet been conducted. The purpose of this paper is to provide a  systematic mathematical framework for an EIT-based flexible pressure sensor.

Our rigorous mathematical analysis is based on the consideration of  a simple model of an EIT-based pressure-sensor using a thin, flexible conductive membrane whose  electrical conductance is directly related to pressure-induced deformation. We assume that the conductive membrane is stretched over a fixed frame and has a number of electrodes  placed on its  boundary as shown in Figure \ref{membrane}. As in a standard EIT system, we use all adjacent pairs of electrodes to inject currents and  measure induced boundary voltages between all neighboring pairs of electrodes to get a current-voltage data set, which is a discrete version of a Neumann-to-Dirichlet map. The current-voltage data can probe any external pressure loaded onto the membrane, because the pressure-induced surface deformation results in a change of the current density distribution over the surface, which leads to a change of the current-voltage data. Hence, the change in the current-voltage data can be viewed as a non-linear function of pressure. The inverse problem in this model is to identify the pressure (equivalently the surface deformation) from the boundary current-voltage data.
\begin{figure}[!h]
\centering
\includegraphics[scale=.4]{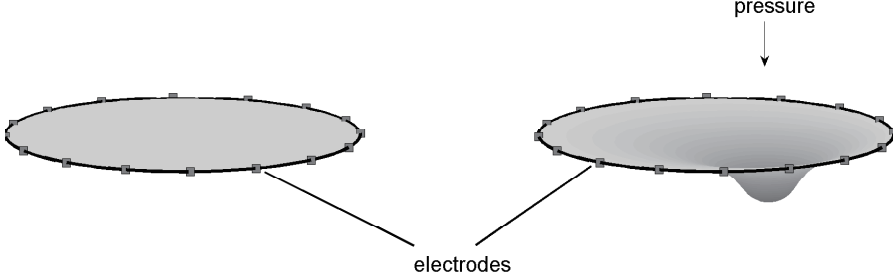}
\vspace{-1cm}\caption{\label{membrane}Conductive membrane attached electrodes on the boundary in the absence of pressure(left) and presence of pressure(right).}
\end{figure}

This paper provides a derivation of an EIT-based pressure-sensing model, which describes the explicit relationship between the measured current-voltage data and the pressure. The mathematical model is associated with an elliptic partial differential equation (PDE) with an anisotropic coefficient, which comes from the pressure-induced surface deformation. To be precise, let $\Om$ be a two-dimensional domain with a smooth boundary $\partial \Omega$ occupying the un-deformed membrane in the absence of any pressure. We denote the standard Sobolev space of order $s$ as $H^s(\partial \Omega)$.

Let $p$ be the pressure and $w_p$ be the solution of
  \begin{equation} \label{Model2}
\left\{
\begin{array}{cl}
\nabla\cdot\left(\frac{1}{\sqrt{1+|\nabla w_p|^2}}\nabla w_p\right)=p &\mbox{ in } \Omega,\\
w_p=0 &\mbox{ on } \p\Om.\\
\end{array}
\right.
\end{equation}

Under pressure $p$, the current-voltage data $(g,f)\in H^{-\f{1}{2}}(\p\Om)\times H^{\f 12}(\p\Om)$ are dictated by $f=u_p|_{\p\Om}$ with $u_p$ being the solution of the elliptic PDE,
  \begin{equation} \label{Model}
\left\{
\begin{array}{l}
\nabla\cdot\left( \gamma_p \nabla u_p\right)=0 \quad \mbox{ in }\Om, \\
(\gamma_p\na u_p)\cdot\mathbf{\nu}|_{\p\Om}~=g, \quad \int_{\partial \Om} u_p =0,
\end{array}
\right.
\end{equation}
where  $$\gamma_p=I -\frac{1}{1+|\nabla w_p|^2}\nabla w_p\nabla w_p^T.$$
Here,
$I$ is the identity matrix, the superscript $T$ denotes the transpose, $\nu$ the unit outward normal vector to $\p\Om$, and  $\int_{\p\Om} g=0$.

The standard  Neumann-to-Dirichlet map $\Lambda_{\gamma_p}$ is defined by $\Lambda_{\gamma_p}(g)=u_p|_{\p\Om}$ with $u_p$ being the solution of \eref{Model}.
We cannot invert the map $\gamma_p\mapsto \Lambda_{\gamma_p}$ with the existing EIT reconstruction methods because of the well-known non-uniqueness result of the inverse problem: there are infinitely many anisotropic coefficients $\gamma$ such that $\Lambda_{\gamma_p}=\Lambda_{\gamma}$.  Hence, we must consider the constrained inverse problem of recovering anisotropic coefficient within the set of coefficient tensors associated with pressures.  Taking account of the fact that two different pressures $p$ and $-p$ produce the same Neumann-to-Dirichlet map, we need to impose a proper constraint on pressures.

Next, we propose a pressure reconstruction method with the standard $N$-channel EIT system. Owing to the quadratic structure of  $\nabla w_p\nabla w_p^T$ in $\gamma_p$, we cannot expect a linearized reconstruction method for $p$,  even assuming that pressure is small. Regarding $p$ as a piecewise constant function $p=\sum_k p_{k}\chi_{T_k}$,  through  the standard discretization of the domain into small elements, $T_k$, the inverse problem can be approximated by solving a large linear system with a large number of unknowns involving all possible products $p_{k}p_{\ell}$. (Here, $\chi_{T_k}$ is the indicator function of  $T_k$.) Given that  most of the columns of the matrix  have relatively small effect  on the data,  we consider  a reduced linear system by eliminating most of the columns.  Various numerical simulations verify the feasibility of the reconstruction algorithm.

In section~\ref{sec:framework}, we formulate the mathematical model for the EIT-based membrane pressure sensor, and present uniqueness results.
In section~\ref{sec:recon_method}, we propose a reconstruction method to recover the pressure. In section~\ref{sec:numerical_results}, we develop a reconstruction algorithm based on sensitivity analysis, and validate the algorithm by numerical simulation results.

This mathematical study of an EIT-based flexible pressure sensor is in an early stage. The proposed mathematical model requires the assumption of incompressibility, whereas there are many flexible materials that are not incompressible. Constructing a mathematical model that includes compressibility will be a future research topic.
\section{Mathematical Framework}\label{sec:framework}
\subsection{Formulation of the forward problem}
Assume that a thin conductive membrane at rest occupies a two-dimensional bounded domain $\Om\subset\R^2$ with a smooth boundary $\p\Om$. Here, the thickness of the membrane is uniform. Assume that the conductivity of the membrane is homogeneous.  Let  $\Om_{d_0}=\{ \x\in\Om~: \mbox{dist}(\x,\p\Om)> d_0\}$ with $d_0>0$. Assume that a pressure $p$ lies in the set
$$\mathfrak{S}=\{ p\in L^\infty(\Om)~:  \| p\|_{L^\infty(\Om)} < \alpha, ~\mbox{supp} (p) \subset \Om_{d_0} \},$$ where $\alpha$ is a positive number. (The assumption $ \| p\|_{L^\infty(\Om)} < \alpha$ is only used  to  guarantee existence and uniqueness of the prescribed mean curvature equation \eref{young-laplace} which will be discussed later.)  When the pressure $p$ is loaded on $\Om$, it produces a displacement of the membrane.  The displacement at $\x=(x,y)\in\Om$ from its rest position is denoted by $w_p(\x)$, and the deformed two-dimensional surface can be expressed as
\begin{equation}\label{om_p}
\Om_p=\{(\x, w_p(\x))~:~\x\in\Om\}.
\end{equation}
Here, the boundary $\p\Om$ of the membrane is fixed so that there is no displacement on the boundary. Because the membrane undergoes deformation to reduce the area change caused by pressure $p\in \mathfrak{S}$, $w_p$ satisfies the prescribed mean curvature equation,
\begin{equation}
\label{young-laplace}
\left\{
\begin{array}{clc}
\nabla\cdot\left(\frac{1}{\sqrt{1+|\nabla w_p|^2}}\nabla w_p\right)&=p &\mbox{in } \Omega,\\
w_p&=0 &\textrm{ on } \partial\Omega .\end{array}\right.
\end{equation}
Problem \eref{young-laplace} has a unique solution for $p\in \mathfrak{S}$ with $\alpha$ being sufficiently small such that, for any measurable subset $E$ of $\Om$, $\int_E p\,d\x$ is smaller than the perimeter of $E$
 \cite{Bernstein1910,Finn1965,Giaquinta1974,Giusti1976,Giusti1978,Giusti1984}.

Let $H_{\diamond}^{-1/2}(\p\Om):=\{g\in H^{-1/2}(\p\Om) :\langle g, 1 \rangle =0\}$ with $\langle \, , \, \rangle$ denoting the duality pair between $H^{-1/2}(\partial \Omega)$ and $H^{1/2}(\partial \Omega)$. Let $H^1(\Omega)$ be defined by  $H^1(\Om) = \{ u \in
L^2(\Om): \nabla u \in L^2(\Omega)\}$ and let $H_0^{1}(\Om)$ be the set of functions in $H^1(\Om)$ with trace zero on $\partial \Omega$.

To extract EIT-data for pressure-sensing, we inject a current of $g\in H^{-1/2}_{\diamond}(\p\Om)$ into the membrane $\Om_p$. In the absence of the pressure ($p=0$), the induced potential due to the injection current of $g$ is the solution $u_0\in H^1(\Om)$ of the following Neumann problem
\begin{equation} \label{NeumannBVP}
\left\{
\begin{array}{ll}
~~~\Delta  u_0 =0\q \quad \mbox{ in }\Omega,\\
\nu\cdot\na u_0=g ~~~~\mbox{ on } \p\Om.
\end{array}
\right.
\end{equation}
In the presence of the pressure ($p\neq 0$), the induced potential $v_p$ is now defined on the deformed surface $\Om_p$, and is governed by
\begin{equation} \label{NeumannBVP-2}
\left\{
\begin{array}{ll}
\na_S\cdot\left( \frac{1}{\sqrt{1+|(\na w_p)\circ \pi_\x|^2}}\na_S v_p\right)=0\q \mbox{ on }\Omega_p,\\
~~~\nu_S\cdot\frac{1}{\sqrt{1+|(\na w_p)\circ \pi_\x|^2}}\na_S  v_p =g~~~~ \mbox{ on } \p\Om_p=\p\Om,
\end{array}
\right.
\end{equation}
where $\na_S$ is the surface gradient, $\nu_S$ is the outward unit normal vector to the boundary $\p\Om_p$, and $\pi_\x$ is the projection map $\Om_p\rightarrow \Om$ defined by $\pi_\x( x,y,z)=(x,y)$.

For the derivation of equation \eref{NeumannBVP-2}, we use  the concept of surface
conductivity \cite{McAllister1991}, while regarding the thin membrane as a two-dimensional surface, because the induced current density along the thin membrane can be viewed as a tangent vector field on the surface. If the deformed membrane is uniform in thickness, the resulting potential $v_p$ satisfies the surface Laplace equation, $\Delta_S v_p=0$, along the surface with $\Delta_S$ being the Laplace-Beltrami operator.  However, under the incompressibility assumption, the thickness of the membrane varies, and so does the surface conductivity.  As a small area, $\Delta x\Delta y$, is changed to $\sqrt{1+|\na w_p|^2}\Delta x\Delta y$, the thickness is approximately reduced by a factor of $\sqrt{1+|\na w_p|^2}$, as is the surface conductivity.

Define $\Upsilon_p : H^{-1/2}_{\diamond}(\p\Om)\rightarrow H^{1/2}_{\diamond}(\p\Om)$ by
\begin{equation}\label{NtD}
\Upsilon_p(g)=v_p|_{\p\Om_p}.
\end{equation}
 The pair $(g,\Upsilon_p(g))$ is called  current-to-voltage pair. Here, $H_{\diamond}^{1/2}(\p\Om):=\{g\in H^{1/2}(\p\Om) :\int_{\partial \Omega} g  =0\}$.

Apparently, the voltage difference, $ \Upsilon_p(g)-\Upsilon_0(g)$, reflects the information of the displacement. Therefore, it  is possible to recover $p$ from several pairs $(g^j, \Upsilon_p(g^j)),~ j=1,2,\cdots, N$. The inverse problem is to reconstruct $p$ and $w_p$ from the boundary voltage-to-current data $(g^j, \Upsilon_p(g^j)),~ j=1,2,\cdots, N$.

There are serious difficulties in solving the inverse problem, because the potential $v_p$ in (\ref{NeumannBVP-2}) is defined on the unknown surface $\Om_p$ in  three dimensions, and the measured data,  $(g^j, \Upsilon_p(g^j))$, are given on the boundary of the two-dimensional domain $\Om$. The relation between the surface, $\Om_p$, and the data,  $(g^j, \Upsilon_p(g^j))$, is too complicated to handle the inverse problem.
To deal with these difficulties, we introduce the following function defined in the known two-dimensional domain, $\Om$, as
\begin{equation}\label{proj}
 u_p(\x)= v_p (\x,w_p(\x))\q\mbox{for }~\x\in\Omega.
 \end{equation}
The following theorem provides a governing equation for $u_p$, through which the relationship between current and voltage can be understood.
\begin{theorem}
The function $u_p$ in (\ref{proj}) is dictated by the following elliptic equation
\begin{equation} \label{proj_con}
\left\{
\begin{array}{cl}
\nabla\cdot\left( \gamma_p\nabla u_p\right)=0 \quad \mbox{ in }~\Om, \\
(\gamma_p\na u_p)\cdot\mathbf{\nu}~=g\quad \mbox{ on }~\p\Om,
\end{array}
\right.
\end{equation}
where $\gamma_p$ is a symmetric positive definite matrix given by
\begin{align}\gamma_p=I -\frac{1}{1+|\nabla w_p|^2}\nabla w_p\nabla w_p^T~~\textrm{in }~\Om.
\label{gamma1}
\end{align}
\end{theorem}
\begin{proof} Let $v^{ext}_p$ denote the extension of $v_p$ such that $v^{ext}_p(\x,z)=v_p(\x,w_p(\x))$ for all $z\in \R$ and $\x\in\Om$. Then, the surface gradient $\na_Sv_p$ can be expressed as
\begin{align}
\na_S v_p =\na_3 v^{ext}_p-(\na_3 v^{ext}\cdot\n_S)\n_S\q\q\mbox{on }~ \Om_p\label{eq1},
\end{align}
where $\n_S=(\na w_p, -1)/\sqrt{1+|\na w_p|^2}$ is the unit downward normal vector to the surface $\Om_p$ and $\na_3=\left(\f{\p}{\p x }, \f{\p}{\p y },\f{\p}{\p z }\right)$ is the three-dimensional gradient.
Since $\p_z v^{ext}_p=0$, from \eref{eq1} a direct computation yields
\begin{align}\hspace{0cm}\label{grad1}
\na_S v_p
&=\frac{1}{1+|\na w_p|^2}\left(\begin{array}{cc}
1+(\p_y w_p)^2 & -(\p_x w_p)(\p_y w_p) \\
-(\p_x w_p)(\p_y w_p) & 1+(\p_x w_p)^2\\
\p_x w_p& \p_y w_p
\end{array}\right)
  \na  u_p\nonumber \\
 &=\left(\gamma_p\na u_p~,~ \frac{1}{1+|\na w_p|^2}\na w_p\cdot\na u_p\right)^T .
\end{align}
Here, we used the fact that $\na v_p^{ext}=\na u_p$.
The surface divergence of the tangential vector field $\frac{\na_S v_p}{\sqrt{1+|\na w_p|^2}}$ is written as
\begin{equation}\label{surf_div}
\na_S\cdot\left(\frac{\na_S v_p}{\sqrt{1+|\na w_p|^2}}\right)=\left(\na_3\times\left(\n_S\times\frac{\na_S v_p}{\sqrt{1+|\na w_p|^2}}\right)\right)\cdot\n_S.
\end{equation}
It follows from the vector identity, $(\na_3\times \A)\cdot \B=\na_3\cdot(\A\times\B)+(\na_3\times\B)\cdot\A$ that \eref{surf_div} can be expressed as
{\small
\begin{align*}
\na_S\cdot\left(\frac{\na_S v_p}{\sqrt{1+|\na w_p|^2}}\right)
&=\na_3\cdot\left(\left(\n_S\times\frac{\na_S v_p}{\sqrt{1+|\na w_p|^2}}\right)\times \n_S\right)+(\na_3\times\n_S)\cdot\left(\n_S\times\frac{\na_S v_p}{\sqrt{1+|\na w_p|^2}}\right)\\
&=\na_3\cdot\left(\frac{\na_S v_p}{\sqrt{1+|\na w_p|^2}}\right)+\f{\na_Sv_p}{\sqrt{1+|\na w_p|^2}}\cdot\left((\na_3\times\n_S)\times\n_S\right).
\end{align*}
}
Replacing $\na_S v_p$ with \eref{grad1}, we obtain
{\small
\begin{align*}
\na_S\cdot\left(\frac{\na_S v_p}{\sqrt{1+|\na w_p|^2}}\right)
&=\na\cdot\left(\frac{\gamma_p\na u_p}{\sqrt{1+|\na w_p|^2}}\right)-(\gamma_p \na u_p)\cdot\na\left(\frac{1}{\sqrt{1+|\na w_p|^2}}\right)\\
&=\frac{1}{\sqrt{1+|\na w_p|^2}}\na\cdot(\gamma_p\na u_p).
\end{align*}
}
Then, $\na_S\cdot\left(\frac{\na_Sv_p}{\sqrt{1+|\na w_p|^2}}\right)=0$ implies $\na\cdot\left(\gamma_p\na u_p\right)=0$, and $\gamma_p$ has the positive eigenvalues $1$ and $1/(1+|\na w_p|^2)$. This completes the proof.
\end{proof}

\subsection{Unique determination of the pressure support}

 We have seen that the displacement, $w_p$, and the current-voltage data, $(g,\Upsilon_p(g))$, are involved in (\ref{proj_con}) with the anisotropic coefficient, $\gamma_p$. In this subsection, we prove that the current-voltage data uniquely determine the pressure support.
To do so, we need to investigate the inverse problem of determining $\gamma_p$ from the current-voltage data. An anisotropic coefficient is uniquely determined by  the current-voltage data up to a diffeomorphism that fixes the boundary. For any diffeomorphism, $\Phi:\Om\rightarrow \Om$ with $\Phi|_{\p\Om}$, being the identity map,
 $u_p\circ\Phi^{-1}$ satisfies
\begin{equation}
\left\{\begin{array}{cc}
&\na\cdot(\gamma_p^{\Phi}\na u_p\circ\Phi^{-1})=0~\textrm{ in }~\Om,\\
& \gamma_p^{\Phi}\na u_p\circ\Phi^{-1}|_{\p\Om}=\gamma_p\na u_p|_{\p\Om},
\end{array}\right.
\end{equation}
 where
 $\gamma_p^{\Phi}$ is a $2\times 2$ symmetric matrix-valued function given by
\begin{equation}\label{diffeo}
\gamma_p^{\Phi}\circ\Phi(\x)=\frac{D\Phi(\x)\gamma_p(\x) D\Phi(\x)^T}{|\det( D\Phi(\x))|}~~~\textrm{ for } \x\in\Om,
\end{equation}
 where $D\Phi$ is the Jacobian of $\Phi$ and $\det$ denotes the determinant.
This means that  two different $\gamma_p$ and $\gamma_p^{\Phi}$ produce the same Neumann-to-Dirichlet map. Conversely, the Neumann-to-Dirichlet map can determine the tensor $\gamma$ up to the diffeomorphism, $\Phi:\Om\to \Om$, where $\Phi|_{\p\Om} =I$ provided that $\gamma$ is approximately constant \cite{Sylvester1990}. In our model,  $\gamma_p$ is only involved in  the scalar  $w_p$, and it is possible to determine $\gamma$ within the set $\Gamma_{\mathfrak{S}}:=\{ \gamma_p~:~ p\in \mathfrak{S}\}$ uniquely provided $\alpha$ is sufficiently small. Note that two different pressures $p$ and $-p$ produce the same coefficient $\gamma_{p}=\gamma_{-p}$.

We must consider the constrained inverse problem of recovering the anisotropic coefficient within the set $\Gamma_{\mathfrak{S}}$ from the current-voltage data.
Let us introduce the {\it outer support}  of $\mathfrak{S}$, denoted by ${\rm supp}_{\p\Om} (p)$;
 for $\x\notin {\rm supp}_{\p\Om} (p)$, there exists an open and connected set $U$ such that $\x\in U$, $\Om\setminus \Om_{d_0}\subset U$, and $p|_{U}=0$ \cite{Gebauer2008,Kusiak2003}.
\begin{theorem}\label{thm_unique}  For $p\in \mathfrak{S}$,   $\Lambda_{\gamma_{p}}$ determines  $\mbox{\rm supp}_{\p\Om}( p)$ uniquely.
\end{theorem}
\begin{proof} Let $p_1,p_2\in \mathfrak{S}$. We assume  $\Lambda_{\gamma_{p_1}}=\Lambda_{\gamma_{p_2}}$. We need to prove that $\mbox{\rm supp}_{\p\Om}( p_1)=\mbox{\rm supp}_{\p\Om}( p_2)$.  From $\Lambda_{\gamma_{p_1}}=\Lambda_{\gamma_{p_2}}$, it follows that $\gamma_{p_1}=\gamma_{p_2}$ on the boundary $\p\Om$  \cite{Sylvester1990}. From \eref{gamma1}, we have
$$
\frac{1}{1+|\nabla w_{p_1}|^2}\nabla w_{p_1}\nabla w_{p_1}^T=\frac{1}{1+|\nabla w_{p_2}|^2}\nabla w_{p_2}\nabla w_{p_2}^T\q\mbox{on }~\p\Om.
$$
This leads to the following identity with $c$ (real) being $|c|=1$:
$$
\nabla w_{p_1}=c \nabla w_{p_2} ~ \q\mbox{on } \p\Om.$$
The difference $w_{p_1}-c w_{p_2}$ satisfies
\begin{align}
\na \cdot( A \na (w_{p_1}-c w_{p_2}))&=p_1-c p_2~~\mbox{in }~\Om , \label{elliptic_A}\\
 w_{p_1}-c w_{p_2}&=0 \q\q\q~~\mbox{on }~\p\Om ,\\
\nabla w_{p_1}-c \nabla w_{p_2}&=0 \q\q\q~~\mbox{on }~\p\Om ,
\end{align}
where $A$ is a matrix given by
\begin{equation}\label{A}
A(\x)=\int_0^1 \f{1}{\sqrt{1+\left|W_t(\x)\right|^2}}\left[ I-\f{W_t(\x) W_t(\x)^T}{
1+\left|W_t(\x)\right|^2} \right] dt\q\mbox{for}~\x\in\Om ,
\end{equation}
and $W_t=t\na w_{p_1}+(1-t)c\na w_{p_2}$. Since the structure of $A$ is the same as $\gamma_p$ in \eref{gamma1}, $A$ is positive-definite and $w_{p_1}-c w_{p_2}$ satisfies the elliptic PDE \eref{elliptic_A}. Hence, by the unique continuation property, it follows that
\begin{equation}\label{bound1}
w_{p_1}(\x)=c w_{p_2}(\x)\q\mbox{for }~ \x\in \Om\setminus \mbox{\rm supp}_{\p\Om}( p_1-c p_2).
\end{equation}

 It remains to prove that $\mbox{supp}_{\p\Om}(p_1)=\mbox{supp}_{\p\Om}(p_2)$. We use Runge approximation argument given by Druskin \cite{Druskin1982} and Isakov \cite{Isakov1988}. For notational simplicity, we denote $D_j:=\mbox{\rm supp}_{\p\Om}( p_j)$ for $j=1,2$. To derive a contradiction, we assume that  $D_1\setminus D_2\neq \emptyset$. Noting that
$$\na\cdot(\gamma_{p_2}\na (u_{p_2}-u_{p_1}))=\na\cdot((\gamma_{p_1}-\gamma_{p_2})\na u_{p_1})~~\mbox{in}~~\Om,$$
if $u_{p_1}=u_{p_2}$ on $\p\Om$, it follows from the assumption $\Lambda_{\gamma_{p_1}}=\Lambda_{\gamma_{p_2}}$ and \eref{bound1} that
$$\int_{\Om}\gamma_{p_2}\na (u_{p_2}-u_{p_1})\cdot\na \varphi=\int_{D_1\cup D_2}(\gamma_{p_1}-\gamma_{p_2})\na u_{p_1}\cdot \na \varphi~~\mbox{for all  }\varphi\in H^1(\Om).$$
This leads to
\begin{equation}\label{runge}
0=\int_{D_1\cup D_2}(\gamma_{p_1}-\gamma_{p_2})\na u_{p_1}\cdot \na u_{p_2}
\end{equation}
for all solutions $u_{p_j}$ to $\na\cdot(\gamma_{p_j}\na u_{p_j})=0$ in $\Om$.
According to the Runge type approximation theorem \cite{Druskin1982,Isakov1988}, we can choose sequences of solutions $u_{p_j}^n$ satisfying $\na\cdot(\gamma_{p_j}\na u_{p_j}^n)=0$ such that
$$
\lim_{n\to\infty}\int_{\Om} (\gamma_{p_1}-\gamma_{p_2}) \na u_{p_1}^n\cdot \na u_{p_2}^n d\x=\infty,
$$
which contradicts  \eref{runge}. This completes the proof.
\end{proof}

\subsection{Unique determination of the pressure in the monotone case}

We now prove the unique determination of the pressure from the current-voltage data in the monotone case.
\begin{theorem}
Let $p_1,p_2$ be in $\mathfrak{S}$. If $p_1 \leq p_2$ in $\Om$ and $\Upsilon_{p_1}=\Upsilon_{p_2}$, then either $p_1=p_2$ or $p_1=-p_2$ in $\Om$.
\end{theorem}

\begin{proof}
To derive a contradiction, we assume that $p_1\neq p_2$ and use exactly the same argument as in the proof of Theorem \ref{thm_unique}. Remember that $w_{p_1} - w_{p_2}$ satisfies the elliptic PDE
$$
\na \cdot( A \na (w_{p_1}- w_{p_2}))=p_1- p_2\q\q \mbox{in }~\Om
$$
with $A$ being defined by (\ref{A}).
From the strong comparison principle, it follows that $$w_{p_1}> w_{p_2}~~\textrm{in}~~\Om.$$
Since $w_{p_j}=0$ on $\p\Om$  for $j=1,2$, we have from Hopf's lemma
   $$\p_\n w_{p_1}< \p_\n w_{p_2}\q\mbox{on }\p\Om.$$
Noting that  $\na w_{p_j}=(\p_\n w_{p_j})\n $ on $\p\Om$, we have
$$
\mbox{either }\q~~|\na w_{p_1}|\neq |\na w_{p_2}|~ \mbox{on }~\p\Om\q~~\mbox{or}\q~~ p_1=-p_2~\mbox{in}~\Om.
$$
Hence, if $p_1\neq -p_2$ in $\Om$, then we have
$$
\gamma_{p_1}=I-\frac{1}{1+|\na w_{p_1}|^2}\na w_{p_1}\na w_{p_1}^T\neq I-\frac{1}{1+|\na w_{p_2}|^2}\na w_{p_2}\na w_{p_2}^T= \gamma_{p_2}~ \mbox{on }~\p\Om. $$
However, this is not possible because  $\Lambda_{\gamma_{p_1}}=\Lambda_{\gamma_{p_2}}$ implies $\gamma_{p_1}|_{\p\Om}=\gamma_{p_2}|_{\p\Om}$ \cite{Kohn-Vogelius}.  This concludes that if $p_1\neq p_2$, then $p_1= -p_2$ in $\Om$, which completes the proof.	 
\end{proof}

It is worth emphasizing  that two different pressures $p_1$ and $p_2$ having the same support can produce the same  displacement near the boundary. More precisely, there exist two different pressures $p_1$ and $p_2$ such that
$$\mbox{\rm supp}_{\p\Om}( p_1)=\mbox{\rm supp}_{\p\Om}( p_2),$$
and
$$w_{p_1}=w_{p_2}~~\textrm{in}~~ \Om\setminus(\mbox{supp}_{\p\Om}(p_1)\cup\mbox{supp}_{\p\Om}(p_2)) .$$

Let $\Om=B_5$ and $D=B_2$ with $B_r$ being the ball of radius $r$ centered at the origin. Consider the following radial symmetric function
\begin{equation}\label{examp_w}
\hspace{0cm}w_{\rho}(\x)=\left\{\begin{array}{cc}
\rho |\x|^3+(-3\rho+\psi'(2)/4) |\x|^2+(-50\rho-25/4\psi'(2))~~&\mbox{if } \x\in D,\\
\psi(|\x|)-\psi(5)~~& \mbox{if }\x\in \Om\setminus D,\\
\end{array}\right.
\end{equation}
where  $\psi(r)=\sqrt{2}\log\left(r+\sqrt{r^2-0.5}\right)$ and $\rho\in\mathbb{R}$.
 A direct computation shows that $w_{\rho}$ satisfies
$$\nabla\cdot\left(\frac{1}{\sqrt{1+|\nabla w_{\rho}|^2}}\nabla w_{\rho}\right)=p_{\rho}$$
where  $p_{\rho}$ is
\begin{equation*}
p_{\rho}(\x)=\left\{\begin{array}{cc}
\f{\p_r^2w_\rho(\x)+ |\x|^{-1}\p_rw_\rho(\x)+|\x|^{-1}w_\rho^3(\x)}{ (1+[\p_r w_\rho(\x)]^2)^{3/2}} ~~&\mbox{if } \x\in D,\\
0~~&\mbox{if }\x\in \Om\setminus D,\\
\end{array}\right.
\end{equation*}
and $\p_r$ is the radial derivative. Hence, $w_{\rho}|_{\Om\setminus D}$ does not change with $\rho$: for every $\rho_1, \rho_2\in \R$,  we have
$$ w_{\rho_1}=w_{\rho_2}~~\textrm{for}~~\x\in\Om\setminus D.$$
This means that in the non-monotone case there are infinitely many $p_{\rho}$ which provide the same displacement near the boundary.

\section{Reconstruction method}\label{sec:recon_method}

\subsection{Measured data: Discrete Neumann-to-Dirichlet map}

We use $N$-channel EIT system in which $N$ electrodes $\{\mathcal{E}_1,\mathcal{E}_2,\ldots,, \mathcal{E}_N\}$ are attached on the boundary $\p\Om$. Let $u_p^j$ be the potential in (\ref{proj_con}) with $g=g^j$ which corresponds to the $j$th injection current using the adjacent pair $\mathcal{E}_j$ and $\mathcal{E}_{j+1}$.
When we inject a current of $I_0$[mA] along the adjacent electrodes $\mathcal{E}_j$ and $\mathcal{E}_{j+1}$, the resulting potential $u_p^j$ satisfies
\begin{equation}\label{Mp-PDE}
\left\{\begin{array}{lc}
\na\cdot(\gamma_p\na u_p^j)=0~~\textrm{in}~\Om,\\
\int_{\mathcal{E}_j}(\gamma_p\na u_p^j)\cdot\nu ds=I_0=-\int_{\mathcal{E}_{j+1}}(\gamma_p\na u_p^j)\cdot\nu ds,\\
(\gamma_p\na u^j_p)\cdot\nu=0~~\textrm{on}~\p\Om\setminus\mathcal{E}_j\cup\mathcal{E}_{j+1} ,\\
\na u_p^j\times \nu =0~~\textrm{on}~\mathcal{E}_j\cup\mathcal{E}_{j+1} .
\end{array}
\right.
\end{equation}
 The $i$th boundary voltage subject to the $j$th injection current is denoted as
\begin{align}\label{mea_data}
V^{i,j}_p= I_0(u_p^j|_{\mathcal{E}_i}-u_p^j|_{\mathcal{E}_{i+1}}) ~~\textrm{for  }i,j=1,2,\ldots,N.
\end{align}
 Here, $\mathcal{E}_{N+1}=\mathcal{E}_1$.
Integration by parts gives the  reciprocity principle:
\begin{equation}\label{recipro}
V^{i,j}_p=\int_{\Om}(\gamma_p\na u_p^i)\cdot\na u_p^jd\x=V^{j,i}_p.
\end{equation}
The boundary voltage (\ref{mea_data}) is assumed to be known. We use it as measurement data for recovery of pressure $p$.

\subsection{Discrepancy minimization problems}

Let $\mathcal{V}^{i,j}$ be the measured data under an applied pressure $p$. Then it follows from (\ref{recipro}) that the pressure $p$ can be obtained by minimizing the discrepancy  functional
\begin{equation}\label{functional}
\mathcal{J}(p)=\sum^N_{i,j=1}\left|\int_{\Om}\left(\gamma_p\na u_p^i\right)\cdot\na u_p^jd\x-\mathcal{V}^{i,j}\right|^2.
\end{equation}
  The inverse problem can be viewed as finding the minimizer of $\mathcal{J}(p)$. Unfortunately, it is numerically difficult to compute the minimizer of $\mathcal{J}(p)$ because  $\mathcal{J}(p)$ is highly non-linear with respect to $p$.

To extract necessary information about $p$ from the data $\mathcal{V}^{i,j}$, we use the voltage difference data
\begin{equation}\label{measure_data}
\mathcal{W}^{i,j}:=\mathcal{V}^{i,j}-u_0^{i,j},
\end{equation}
 where $u_0^{i,j}$ is the data in \eref{mea_data} with $p=0$, {\it i.e.,} the boundary voltage data in the absence of the pressure. With the voltage difference data $\mathcal{W}^{i,j}$, the functional $\mathcal{J}(p)$ in \eref{functional} can be rewritten as
\begin{equation}\label{func_J}
\mathcal{J}(p)=\sum^{N}_{i,j=1}\left|\int_\Om\left( \frac{\na w_p\na w_p^T}{1+|\na w_p|^2}\na u_p^i\right)\cdot \na u_0^jd\x-\mathcal{W}^{i,j}\right|^2.
\end{equation}
The above identity follows from
\begin{align*}
\hspace{0cm}\int_\Om (\gamma_p\na u_p^i)\cdot\na u_p^jd\x-\mathcal{V}^{i,j}&=\int_\Om \na u_p^i\cdot\na u_0^jd\x-\mathcal{V}^{i,j}\\
&=\int_\Om \na u_p^i\cdot\na u_0^jd\x -\int_\Om \na u_0^i\cdot \na u_0^jd\x-\mathcal{W}^{i,j}\\
&=\int_\Om \na u_p^i\cdot\na u_0^jd\x -\int_\Om \left(\gamma_p\na u_p^i\right)\cdot \na u_0^jd\x-\mathcal{W}^{i,j}\\
&= \int_\Om \left((I-\gamma_p)\na u_p^i\right)\cdot \na u_0^jd\x-\mathcal{W}^{i,j}.
\end{align*}

Due to the high non-linearity of the discrepancy functional $\mathcal{J}$ in \eref{func_J}, it is difficult to compute a minimizer $p$ directly. To compute minimizers of \eref{func_J} effectively, we will make use of various approximations. Assume that $|\na w_p|$ is small. We will neglect quantities of fourth order of smallness; for example,
\begin{equation}\label{Wp-approx1}
\frac{\na w_p\na w_p^T}{1+|\na w_p|^2}=\na w_p\na w_p^T +O(|\na w_p|^4)\approx \na w_p\na w_p^T.
\end{equation}
Since
$
\int_\Om \gamma_p \na (u_p-u_0)\cdot\na (u_p-u_0) d\x=\int_\Om (\gamma_p-I) \na u_0\cdot \na (u_p-u_0)d\x,
$
we have
\begin{equation}\label{Wp-approx2}
\|\na (u_p-u_0)\|_{L^2(\Om)}\le C\|\na w_p\|_{L^\infty(\Om)}^2 \|\na u_0\|_{L^2(\Om)}.
\end{equation}
From \eref{Wp-approx1} and \eref{Wp-approx2}, we have
 \begin{equation}\label{ap1}
\hskip -0.2in\int_{\Om}[I-\gamma_p]\nabla u_p\cdot\nabla u_0d\x=\int_{\Om}[\na w_p\na w_p^T]\nabla u_0\cdot\nabla u_0d\x+ O\left(\|\na w_p\|_{L^\infty(\Om)}^6\right).
 \end{equation}
 Neglecting $O\left(\|\na w_p\|_{L^\infty(\Om)}^6\right)$ in \eref{ap1}, the discrepancy functional $\mathcal{J}$ in \eref{func_J} can be approximated as
\begin{equation}\label{lin_J}
\mathcal{J}_1(p)=\sum^{N}_{i,j=1}\left|\int_\Om [\na w_p\na w_p^T]\na u_0^i\cdot \na u_0^jd\x-\mathcal{W}^{i,j}\right|^2.
\end{equation}
Assuming Gaussian measurement noise, we consider the following regularized minimization problem:
\begin{align}\label{lin_regul}
\min_{p} \mathcal{J}^{reg}_1(p)
\end{align}
with
\begin{align}\label{lin_regul2}
\mathcal{J}^{reg}_1(p)=\sum_{i,j=1}^N\left|\int_{\Om} [\na w_p\na w_p^T]\na u_0^i\cdot\na u_0^jd\x-\mathcal{W}^{i,j}\right|^2+\beta\|p\|_{L^2(\Om)}^2,
\end{align}
and $\beta$ being a regularization parameter.

The displacement $w_p$ can be approximated to $v\in H^1_0(\Om)$ with $v$ being the solution of  Possion's equation $\Delta v=p$ in $\Om$,  because $$\int_\Om |\na (w_p-v)|^2 d\x=\int_\Om \left(1-\f{1}{\sqrt{1+|\na w_p|^2}}\right) \na w_p\cdot \na (w_p-v)d\x=O(\|\na w_p\|_{L^\infty(\Om)}^6).$$ With this approximation, the minimization problem (\ref{lin_regul}) can be further simplified as follows: find $(p,v)\in[H^1_0(\Om)]^2$ which minimizes the discrepancy functional
\begin{align}\label{F_prop}
 \mathcal{J}_2(p,v)=\sum_{i,j=1}^N&\left|\int_{\Om}\left([\na v\na v^T]\na u_0^i\right)\cdot\na u_0^jd\x-\mathcal{W}^{i,j}\right|^2\\&+\lambda\int_\Om \left(\frac{1}{2}|\na v|^2+pv\right)d\x +\beta\|p\|_{L^2(\Om)}^2,\nonumber
\end{align}
where $\lambda$ is a positive number. In the next subsection, we minimize the functional defined in \eref{F_prop}  in order to reconstruct $p$.
\subsection{Reconstruction algorithm}\label{subsec:recon_algo}
Based on the simplified discrepancy functional \eref{F_prop}, we propose a pressure image reconstruction algorithm.  We discretize the domain $\Om$ into triangular elements such that $\overline \Omega=\cup_{k=1}^K T_k$, where $T_k$ is a triangular subregion with side length $h<1$. For the approximation of the pressure $p$,
we assume that $p$  is a piecewise constant function contained in the set
$$
\mathcal P_h:=\{ p~:~ p \mbox{ is constant for each } T_k, ~k=1,\cdots, K\}.
$$
We assume that $p\in \mathcal P_h^*:=\mathcal P_h\cap \{ p~:~p=0 \mbox{~in ~} \Om\setminus \Om_{d_0}\}$.
Then, we can express the pressure $p$  by $$p(\x)=\sum_{k=1}^Kp^{(k)}\chi_{T_k}(\x).$$

For each $k=1,\cdots K$, let $v_k$ be  the solution of
\begin{equation}\label{v}
\left\{\begin{array}{rcc}
-\Delta v_k&=\chi_{T_k}~\textrm{ in }\Om,\\
v_k&=0~\textrm{ on } \partial\Om.
\end{array}
\right.
\end{equation}
Then, $v_k$ can be expressed as
\begin{equation}\label{Green-rep}
v_k(\x)=\int_{T_k} G(\x,\y) d\y,
\end{equation}
where $G(\x,\y)$ is the Dirichlet function associated with the domain $\Om$, that is, the solution to
$$
\left\{\begin{array}{rcc}
-\Delta_x G &=\delta_y ~\textrm{ in }\Om,\\
G &=0~\textrm{ on } \partial\Om
\end{array}
\right.
$$
with $\delta_y$ being the Dirac mass at $y$.

If $(p,v)\in \mathcal P_h^*\times H^1(\Om)$ is a minimizer of the functional \eref{F_prop}  with $\lambda=\infty$, $v$ should be given by
$$v=\sum^K_{k=1}p^{(k)}v_k.$$
Regarding the pressure $p$ as a vector $\mathbf{p}=\left( p^{(1)},~p^{(2)},\ldots,~p^{(K)}\right)\in\mathbb{R}^K$, the discretized minimization problem (\ref{F_prop}) for large $\lambda$ can be simplified as
\begin{equation}\label{F_quad}
\mathcal{J}_3(\mathbf{p})=\sum_{i,j=1}^K\left| \mathbf{p}^T\mathbb{Q}^{i,j}\mathbf{p}-\mathcal{W}^{i,j}\right|^2+\beta\|\mathbf{p}\|_2^2,
\end{equation}
where each $\mathbb{Q}^{i,j}$ is a matrix given by
$$\mathbb{Q}^{i,j}=\left(\begin{array}{cccc}
S_{11}^{ij}& S_{12}^{ij} & \cdots &S_{1K}^{ij}\\
S_{21}^{ij}& S_{22}^{ij} & \cdots &S_{2K}^{ij}\\
\vdots&\vdots& &\vdots\\
S_{M1}^{ij}&S_{M2}^{ij}&\cdots&S_{KK}^{ij}
\end{array}\right)_{K\times K}$$
and $S_{k\ell}^{ij}$ is given  by
\begin{equation}\label{S}
S_{k\ell}^{ij}:= \int_{\Om}\left([\na v_k\na v_{\ell}^T]\na u_0^i\right)\cdot\na u_0^jd\x.
\end{equation}
Here, the quadratic term $\mathbf{p}^T\mathbb{Q}^{i,j}\mathbf{p}$ in (\ref{F_quad}) can be viewed as a good approximation of the quantity  ${\displaystyle \int_\Om\left( \frac{\na w_p\na w_p^T}{1+|\na w_p|^2}\na u_p^i\right)\cdot \na u_0^jd\x }$ in terms of the discretized $\mathbf{p}$.

The quadratic form of (\ref{F_quad}) can be converted  to a linear form by introducing the vector ${\mathbf q}=\left(q^{(1)},\cdots, q^{(K^2)}\right)\in \R^{K^2}$ whose components are
$$
q^{(k+\ell(K-1))}=p^{(k)}p^{(\ell)}\q\q\mbox{for } k,\ell=1,\cdots, K.
$$
With this large vector ${\mathbf q}$, the quadratic form $\mathcal{J}_3(\mathbf{p})$ in \eref{F_quad} can be changed into the following linear form:
\begin{equation}\label{F_lin}
\mathcal{J}_4(\mathbf{q})=\left\| \mathbb{S}\mathbf{q}-\mathbf{W}\right\|^2_2+\beta\|\mathbf{q}\|_2^2,
\end{equation}
where $\mathbb{S}$ is $N^2\times K^2$ matrix given by
\begin{equation}\label{senseM}
\left(~(i-1)N+j~,~ (k-1)M+\ell~\right)-\mbox{component of } \mathbb{S} ~~=~~ S_{k\ell}^{ij},
\end{equation}
and
\begin{equation}\label{meausre_vector}
\mathbf{W}=\left( \mathcal{W}^{1,1}\cdots~\mathcal{W}^{1,N}~\mathcal{W}^{2,1}
~~\ldots~\ldots~~\mathcal{W}^{N,1}\cdots~\mathcal{W}^{N,N}\right)\in\mathbb{R}^{N^2}.
\end{equation}

Now, the minimizer of the functional $\mathcal{J}_4$ in (\ref{F_lin}) can be obtained by solving the following linear system:
\begin{equation}\label{lin_eq}
(\mathbb{S}^T\mathbb{S}+\sqrt{\beta}I)\mathbf{q}=\mathbb{S}^T\mathbf{W}.
\end{equation}
Unfortunately, the linear system \eref{lin_eq} is too large to handle; the number of column vectors  of $\mathbb{S}$ is proportional to $h^{-4}$ where $h^2$ is proportional to the mesh size. Hence, we need to eliminate most of the column vectors of the matrix $\mathbb{S}$ whose influence on the data are negligibly small. Noting that $k+\ell(K-1)$-th column vector of $\mathbb{S}$ consists of components $S_{k\ell}^{ij}=\int_{\Om}[\na v_k\na v_{\ell}^T] \na u_0^i\cdot\na u_0^jd\x$, the quantity $\sup_{i,j}|S_{k\ell}^{ij}|$ can be estimated by $\int_{\Om}\left|\na v_k\na v_{\ell}^T\right|d\x$. We will see that the quantity $\int_{\Om}\left|\na v_k\na v_{\ell}^T\right|d\x$ decreases as $\mbox{dist}(T_{k},T_{\ell})$ increases.

Since $p$ is supported in $\Om_{d_0}$, we assume that $\mbox{dist}(T_k, \p\Om)\geq d_0$.
 Using the expression \eref{Green-rep} of $v_k$, we have
\begin{align}\label{green1}
\int_{\Om}\left|\na v_k\na v_{k}^T\right|d\x&= \int_\Om |\na v_k|^2 d\x=
-\int_{\Om} \Delta v_k v_k d\x\nonumber\\
&= \int_{T_k} \int_{T_k} G(\x,\y) d\x d\y \gtrsim h^4\left|\log h\right|,
\end{align}
where the expression  $X\gtrsim Y$ is used to mean that there is a positive constant $C$ independent of $h$ such that $X\ge C Y$. On the other hand, if $\mbox{dist}(T_k,T_\ell)>d_0$, then we have
\begin{align}\label{green2}
\int_{\Om}\left|\na v_k\na v_{\ell}^T\right|d\x &\lesssim\int_\Om \left(\int_{T_k}|\na G(\x,\y)| d\y \int_{T_\ell}|\na G(\x,\y')| d\y'\right) d\x\nonumber\\
& \lesssim h^4\int_\Om |\na G(\x,\z_k)||\na G(\x,\z_\ell)| d\x\nonumber\\
& \lesssim h^4\f{1}{|\z_k-\z_\ell|},
\end{align}
where  $\z_k$ and $\z_\ell$ are the gravitational centers of $T_k$ and $T_\ell$, respectively. From the above estimate, {we observe that  $\sup_{i,j}|S_{k\ell}^{ij}|$ is negligibly small if $\mbox{dist}(T_{k},T_{\ell})$ is large.}

For $\delta\ge 0$, let $\mathbb{S}_\delta$ be the reduced  matrix of $\mathbb{S}$ by eliminating all columns corresponding to pairs $(k,\ell)$ in the set $\mathcal K_\delta$:
$$ \mathcal K_\delta:=\{ (k,\ell)~: \mbox{dist}(T_{k},T_{\ell})> \delta\}.$$
{The parameter $\delta$ indicates the number of columns in $\mathbb{S}$ to be used in order to solve the linear system \eref{lin_eq}. In the case that $\delta< h$ ($\mathcal{K}_\delta=\{(k,\ell)~:~k\neq \ell \}$), we only consider the diagonal terms $(k,k)$, and neglect most of columns in $\mathbb{S}$. When $\delta=\diam(\Om)$, we consider all the pairs $(k,\ell)$ without neglecting any column in $\mathbb{S}$ $(\mathcal{K}_\delta=\emptyset)$. }
Denoting the corresponding reduced vector  of $\mathbf{q}$ by $ \mathbf{q}_\delta$, the large linear system  \eref{lin_eq} can be approximated by the following reduced system:
\begin{equation}\label{reduc_lin}
(\mathbb{S}_\delta^T\mathbb{S}_\delta+\sqrt{\beta}I)\mathbf{q}_\delta=\mathbb{S}_\delta^T\mathbf{W}.
\end{equation}
In our numerical experiments, $\delta$ is chosen to be less than six times the diameter of meshes.

 Based on the minimization problem\eqref{F_prop} with the above reduction strategy, we develop the following pressure reconstruction algorithm\label{recon_algo}.

\begin{center}
\vspace{0.1cm}\begin{minipage}{0.9\linewidth}
\begin{itemize}
 \item[\it Step 1.] Set a reduction parameter $\delta\geq 0$ and compute the reduced sensitivity matrix $\mathbb{S}_\delta$.
 \item[\it Step 2.] Solve \eref{reduc_lin} with a regularization parameter $\beta$ to obtain $\mathbf{q}_\delta$.
 \item[\it Step 3.] Take square-root of the components $\left(p^{(k)}\right)^2$ in $\mathbf{q}_\delta$ and obtain $\mathbf{p}$. When $\mathbf{q}_\delta$ has negative values, we truncate the negative values in $\mathbf{q}_\delta$ before taking the square-root.
 \end{itemize}
 \end{minipage}
\end{center}

\section{Numerical results}\label{sec:numerical_results}
We test the performance of our reconstruction algorithm \ref{recon_algo}. Finite element method (FEM) is used to implement our algorithm.  We perform numerical experiments in two different  domain shapes: a square shaped domain with $K=512$ uniform triangular pixels, a circular shaped domain with $K= 661$ uniform triangular pixels. For each domain, we apply pressures being characteristic functions with different supports, namely $p(x,y)=p_0 \chi_{D}(x,y)$ where $D$ consists of three or four pressured regions  depicted in the first columns of Figures \ref{simu1}-\ref{simu2}.

With these pressure distributions, we compute  the current-voltage data \eref{measure_data} by solving  \eref{young-laplace} and \eref{Mp-PDE} with a 16-channel EIT-system ($N=16$). For the reconstruction of $p$, we use  the reconstruction algorithm \eref{reduc_lin}, described in section \ref{recon_algo}, with the reduction parameter $\delta=5h$, where $h$ is the side length of triangular mesh. We compare its performance with those obtained by the conventional EIT reconstruction method and  the method \eref{reduc_lin} with  $\delta< h$. The size of matrices $\mathbb{S}_\delta$ for various $\delta$ are given in Table \ref{table}. For the numerical simulation, the same regularization parameter $\beta$ is used to solve \eref{reduc_lin}.

 Figure \ref{simu1}-\ref{simu2} (a) and (e) show the true pressure distributions having different supports and the same magnitude. Figure \ref{simu1}-\ref{simu2}(b) and (f) display the conductivity variations measured by conventional EIT-method with respect to (a) and (e), respectively. Figure \ref{simu1}-\ref{simu2}(c) and (g) show the reconstructed pressure distributions by the proposed reconstruct algorithm with $\delta< h$. Figure \ref{simu1}-\ref{simu2}(d) and (h) show the reconstructed pressure distributions by the proposed method.   We observe through Figure \ref{simu1}-\ref{simu2}(c) and (d) that the proposed reconstruction algorithm with the higher reduction parameter $\delta=5h$ provides more better image than $\delta<h$ for the detection of the pressured regions.

We should mention that, in the case when the magnitude of the pressure is sufficiently small, the conventional linearized EIT reconstruction method work well with the selection of a good regularization parameter. However, this regularization method does not work well when the pressure is not small.

 \begin{table}[!ht]
\centering
\caption{\label{table} The matrix size of the reduced sensitivity matrix $\mathbb{S}_\delta$. Higher value of $\delta$ indicates that more columns of the sensitivity $\mathbb{S}$ are used to solve \eref{reduc_lin}. The size of $\mathbb{S}$ is given in the last row ($\delta=\diam(\Om)$). Here, $h$ is the side length of the triangular mesh. }
\footnotesize{
\begin{tabular}{@{}*{7}{c}}
\hline
~$\delta$  & The size of ~$\mathbb{S}_\delta$ (square)\cr
\hline
0 &  $256\times 512$\cr
$5h$ &  $256\times 28018$\cr
$\diam(\Om)$ &  $256\times 262144$\cr
\hline
\end{tabular}\begin{tabular}{@{}*{7}{c}}
\hline
~$\delta$  & The size of ~$\mathbb{S}_\delta$ (circle)\cr
\hline
0 &  $256\times 661$\cr
$5h$ &  $256\times 43814$\cr
$\diam(\Om)$ &  $256\times 436921$\cr
\hline
\end{tabular}
}
\end{table}

\begin{figure}[ht]
\centering
\begin{tabular}{cccc}
\small
\includegraphics[width=0.22\linewidth]{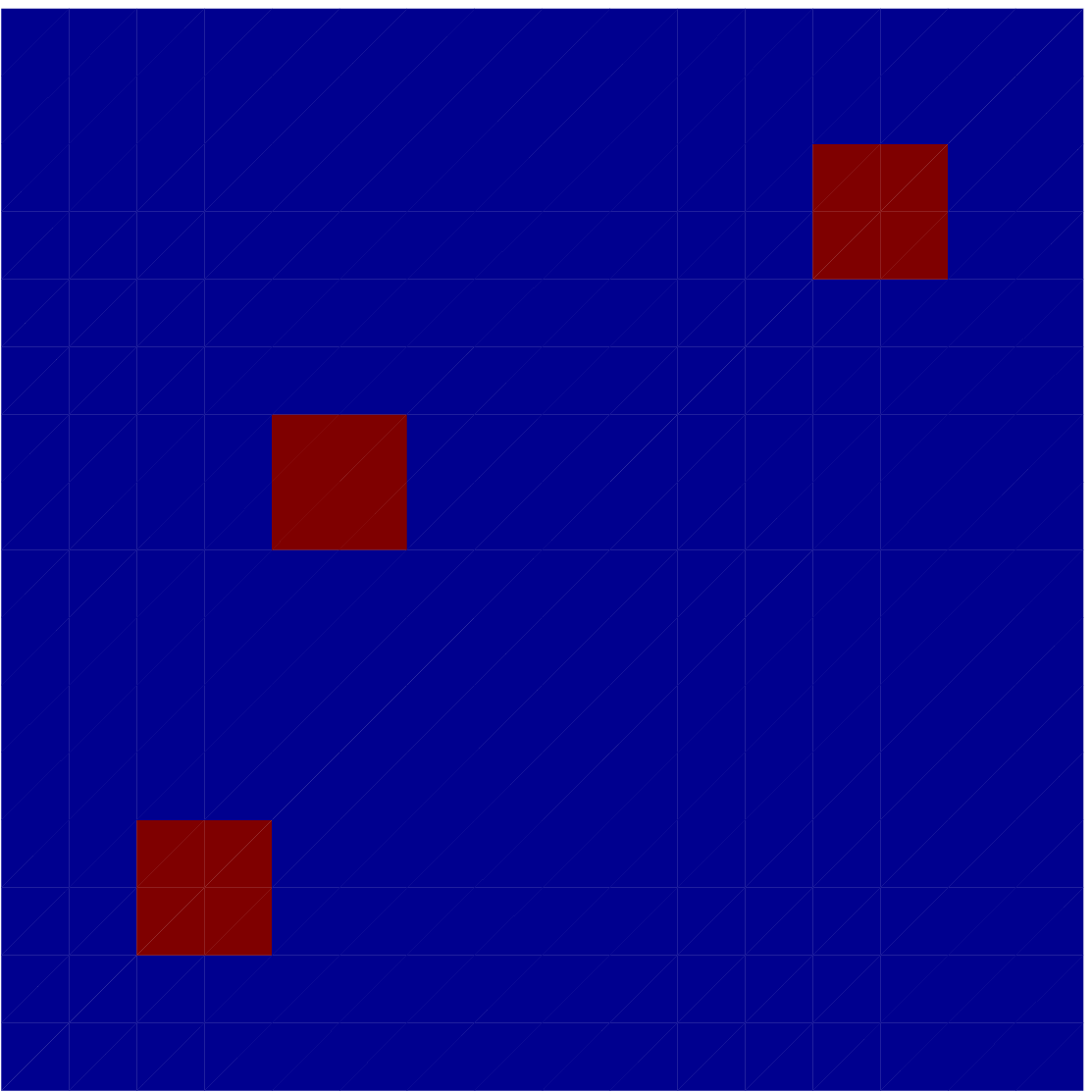}&
\includegraphics[width=0.22\linewidth]{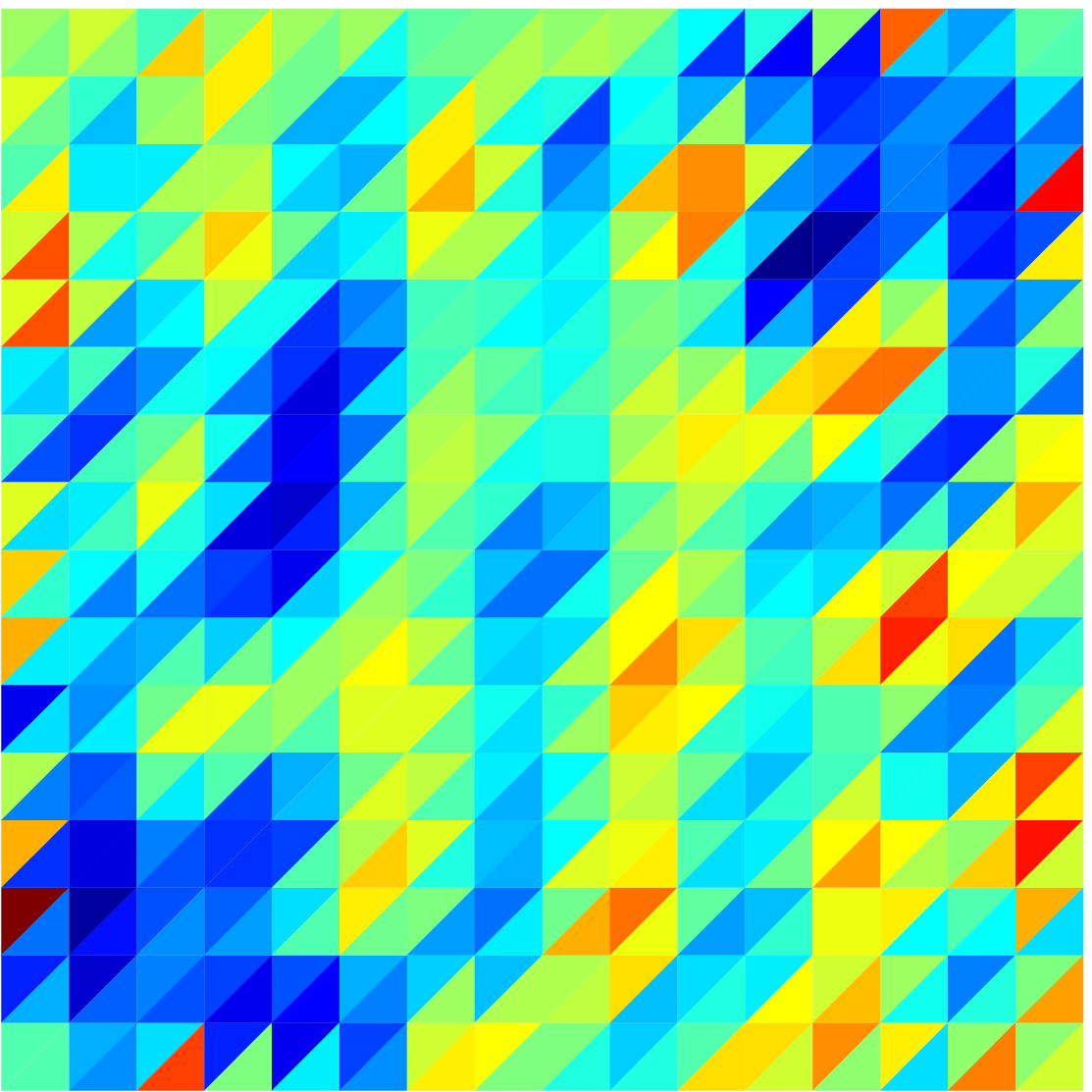}&
\includegraphics[width=0.22\linewidth]{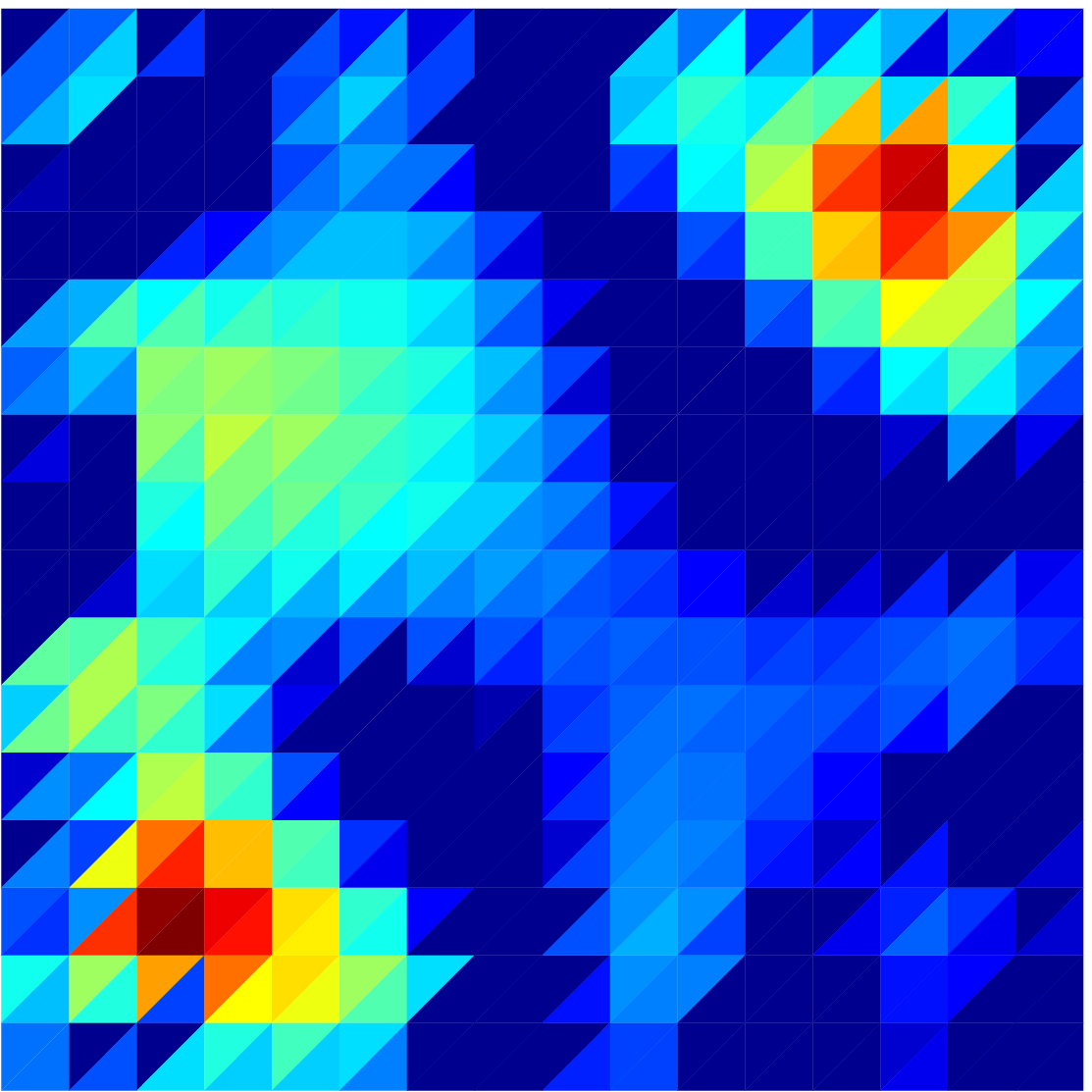}&
\includegraphics[width=0.22\linewidth]{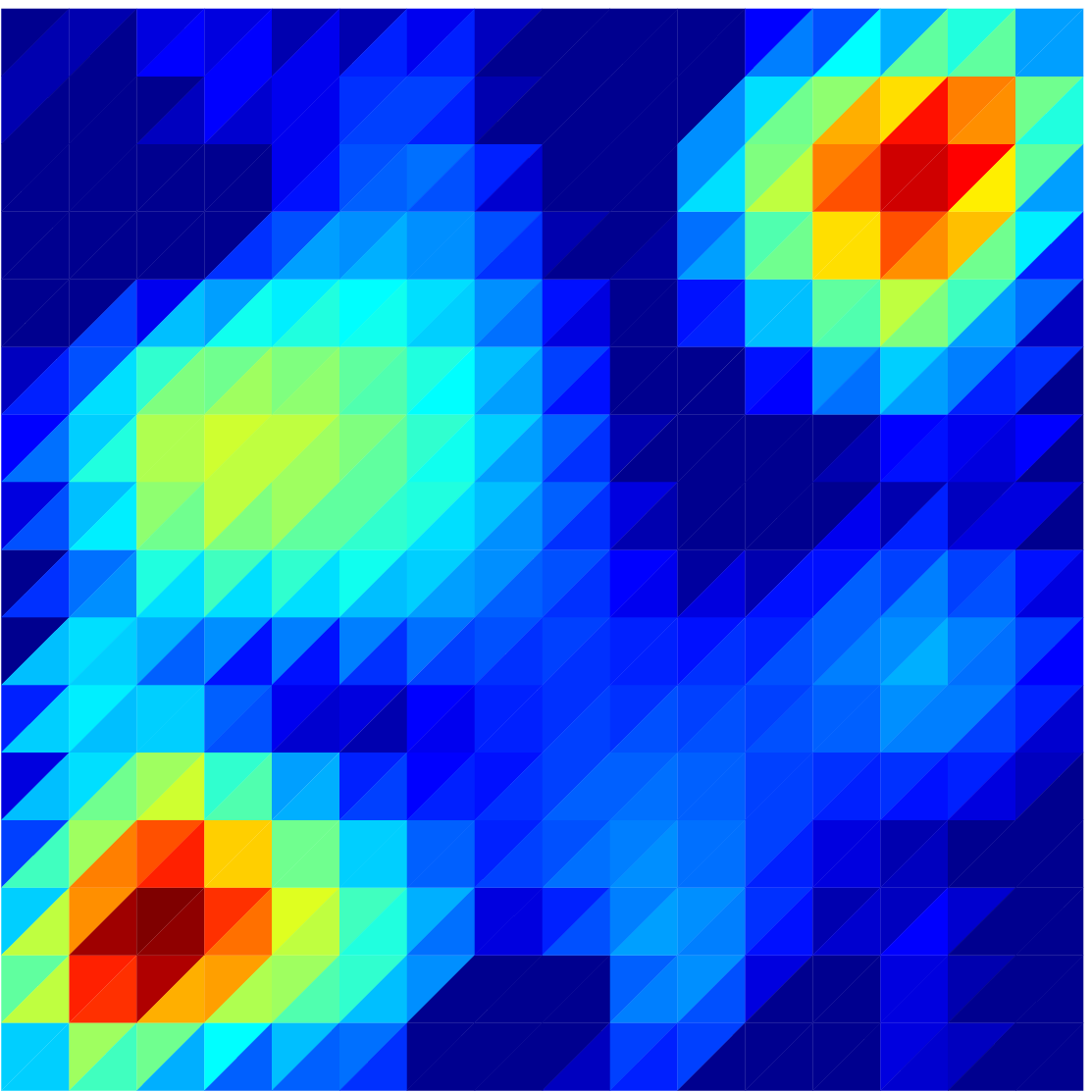}\\
 \begin{tabular}{c}\small(a) true pressure\\ \small  \end{tabular}  &  \begin{tabular}{c}\small(b) conventional\\ \small method \end{tabular} & \begin{tabular}{c}\small(c) reconstruction\\ \small algorithm($\delta< h$)\end{tabular}  &  \begin{tabular}{r}\small(d) proposed\\ \small method \end{tabular} \\
\includegraphics[width=0.22\linewidth]{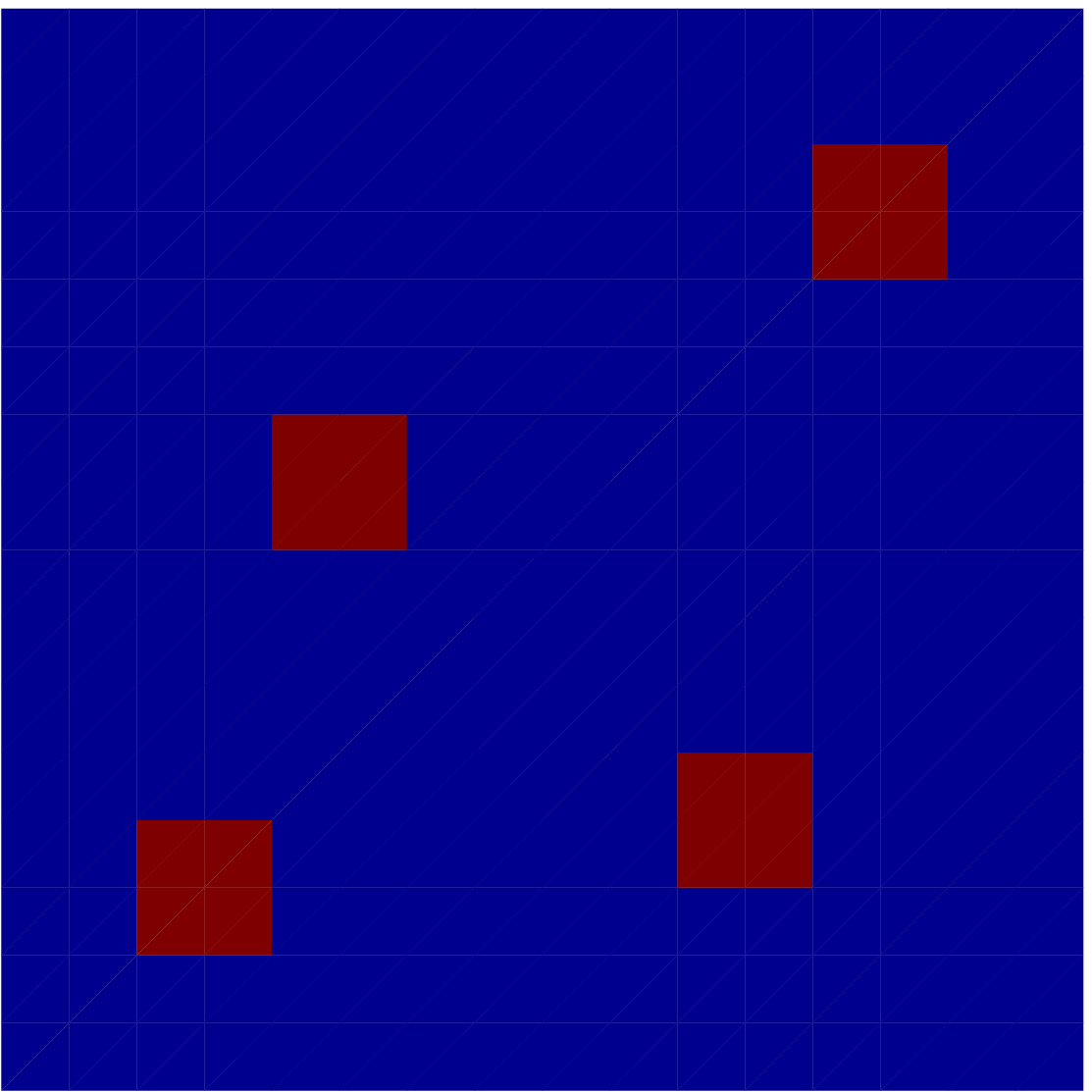}&
\includegraphics[width=0.22\linewidth]{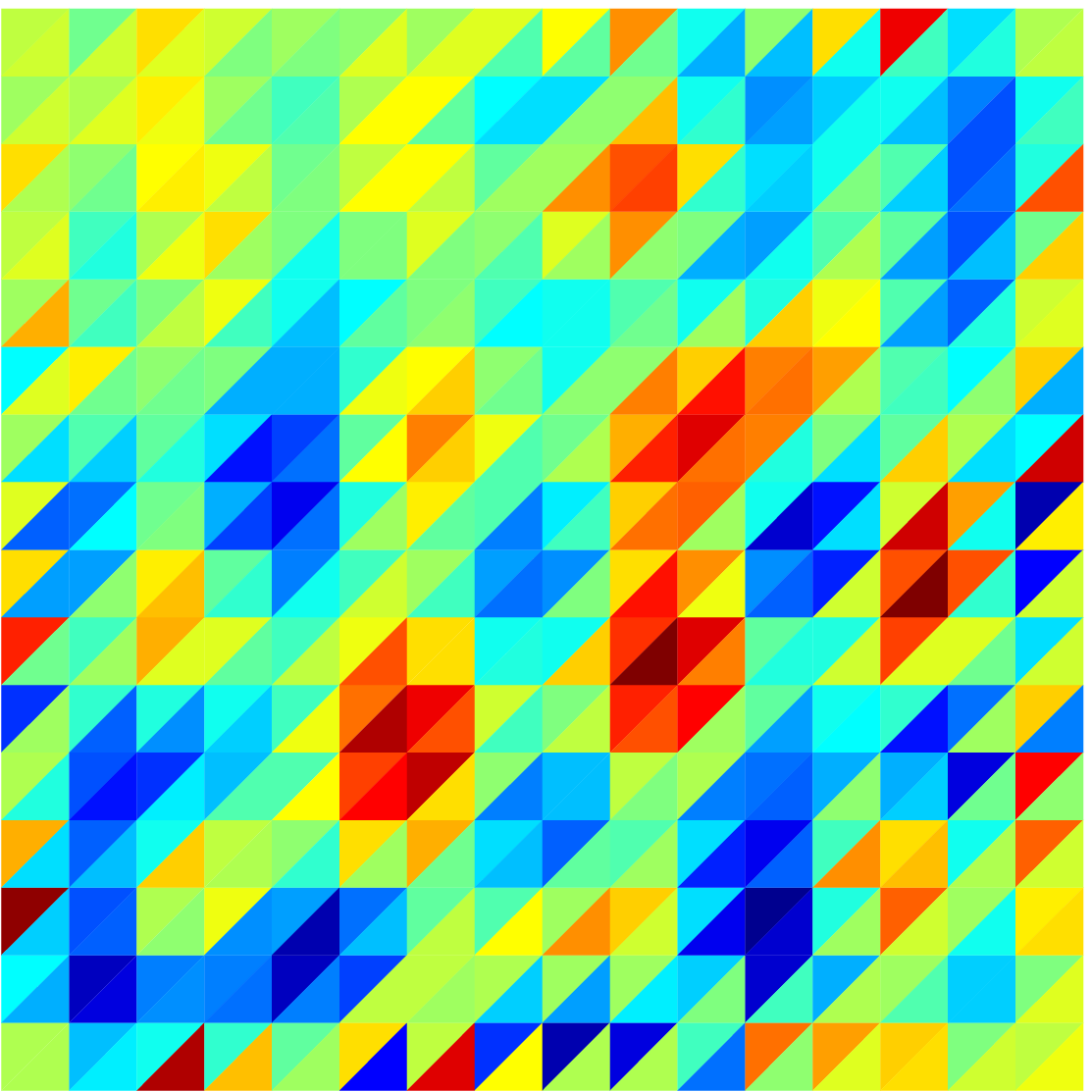}&
\includegraphics[width=0.22\linewidth]{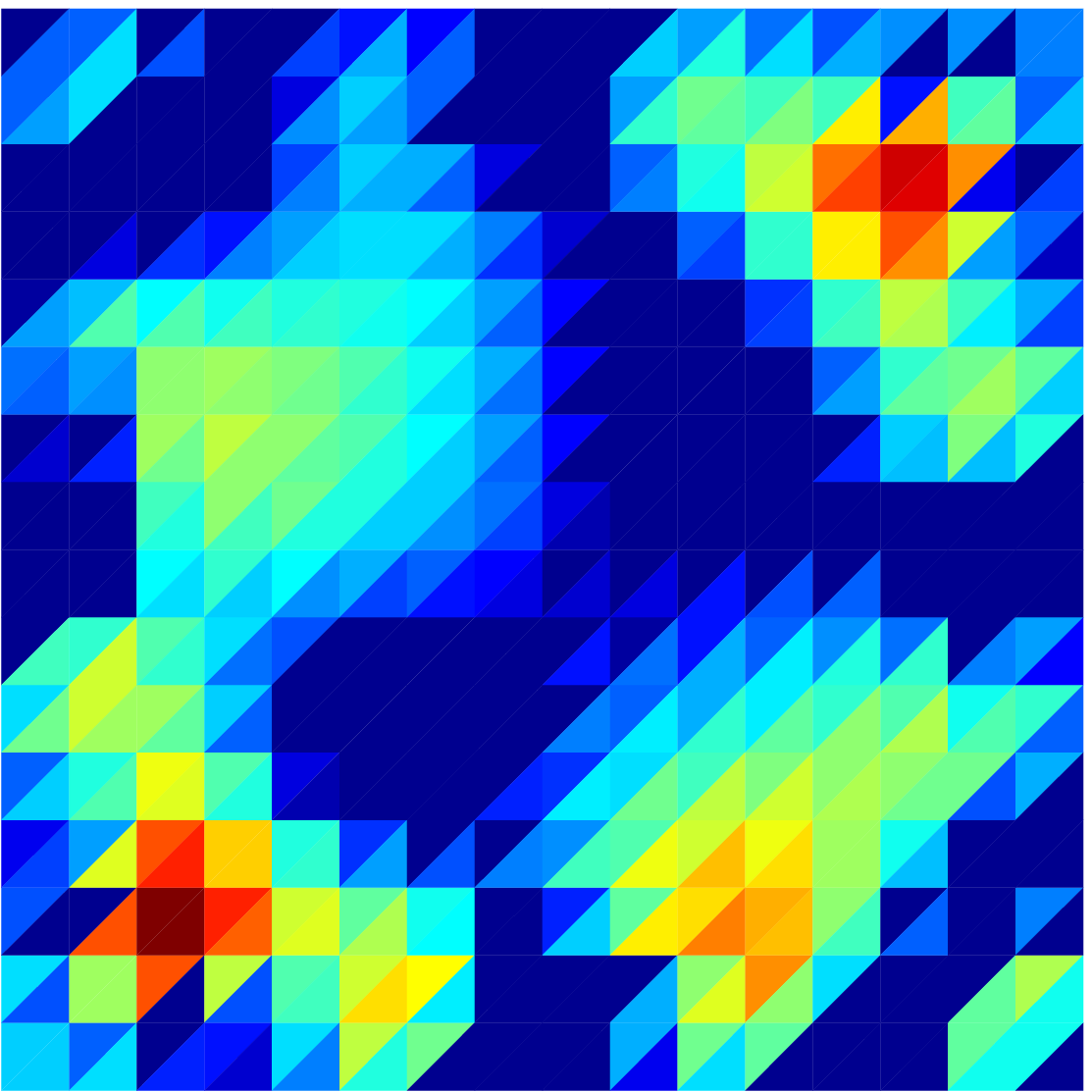}&
\includegraphics[width=0.22\linewidth]{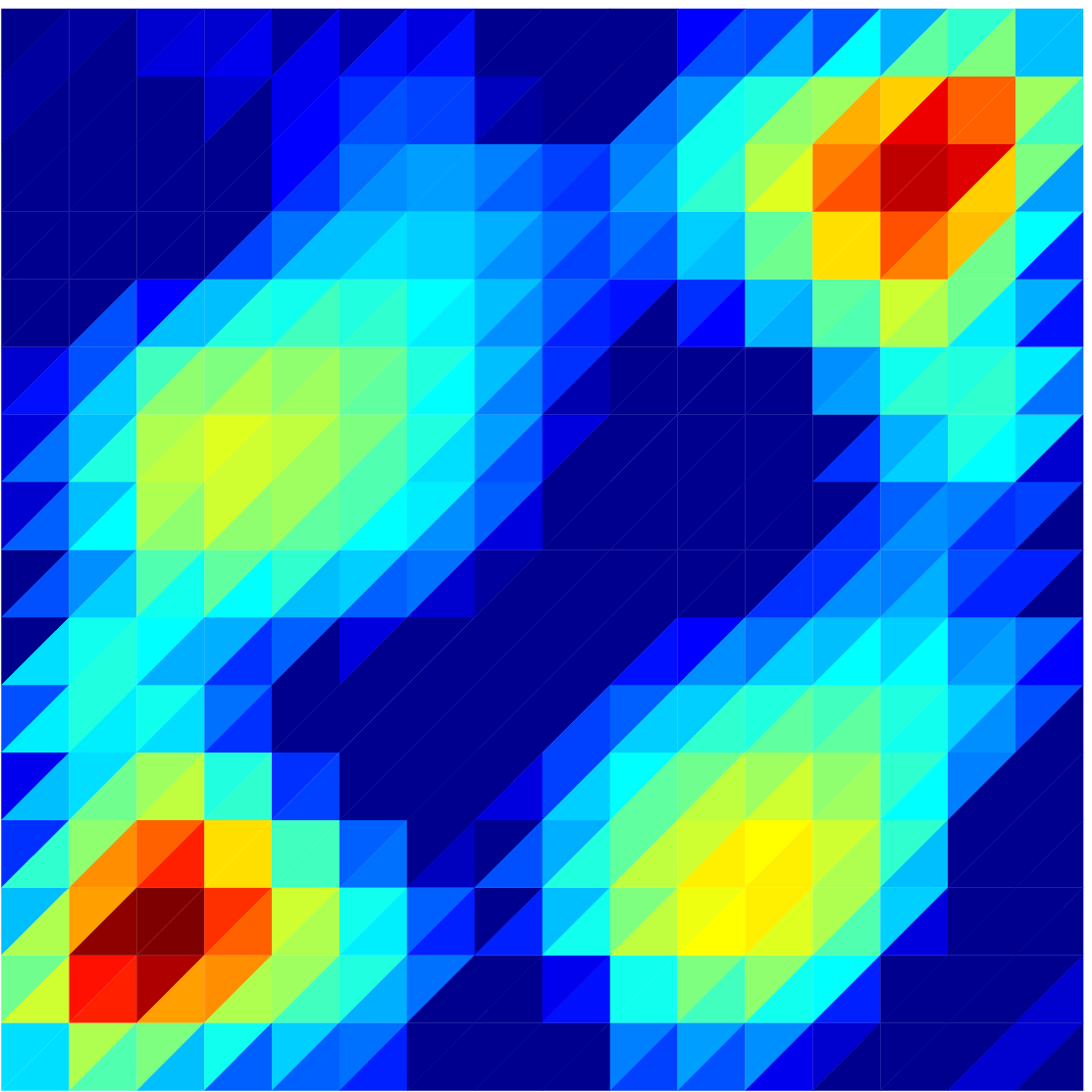}\\
 \begin{tabular}{c}\small(e) true pressure\\ \small  \end{tabular}  &  \begin{tabular}{c}\small(f) conventional\\ \small method \end{tabular} & \begin{tabular}{c}\small(g) reconstruction\\ \small algorithm($\delta< h$)\end{tabular}  &  \begin{tabular}{r}\small(h) proposed\\ \small method \end{tabular} \\
\end{tabular}
\caption{\label{simu1}True pressure distributions(first column), the conductivity variation images by conventional EIT-method(second column), and the reconstructed pressure images using the reconstruction algorithm \ref{recon_algo} with $\delta< h$(third column), $\delta=5h$(fourth column). Here, $h$ is the side length of the triangular mesh.}
\end{figure}
\begin{figure}[ht]
\centering
\begin{tabular}{cccc}
\includegraphics[width=0.22\linewidth]{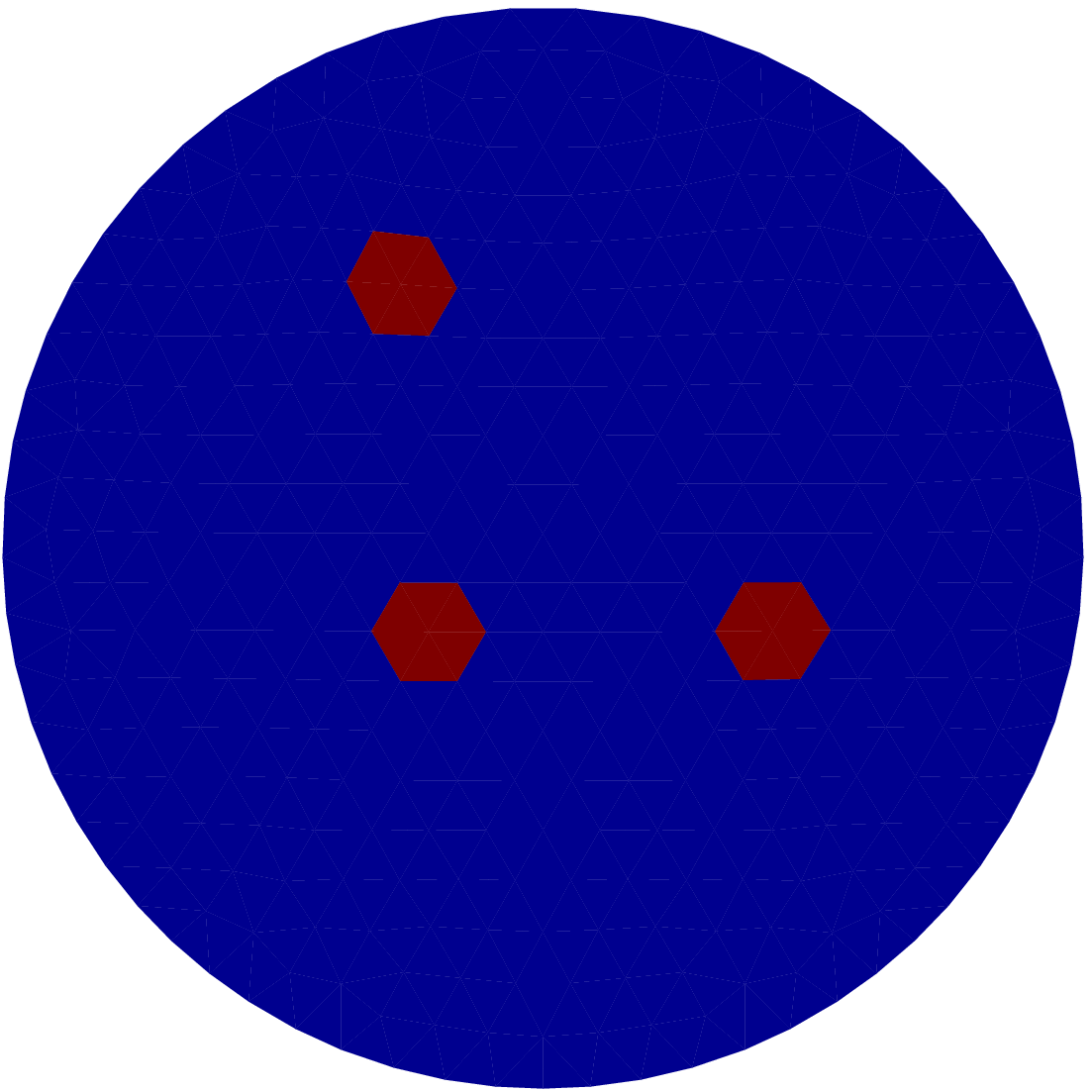}&\includegraphics[width=0.22\linewidth]{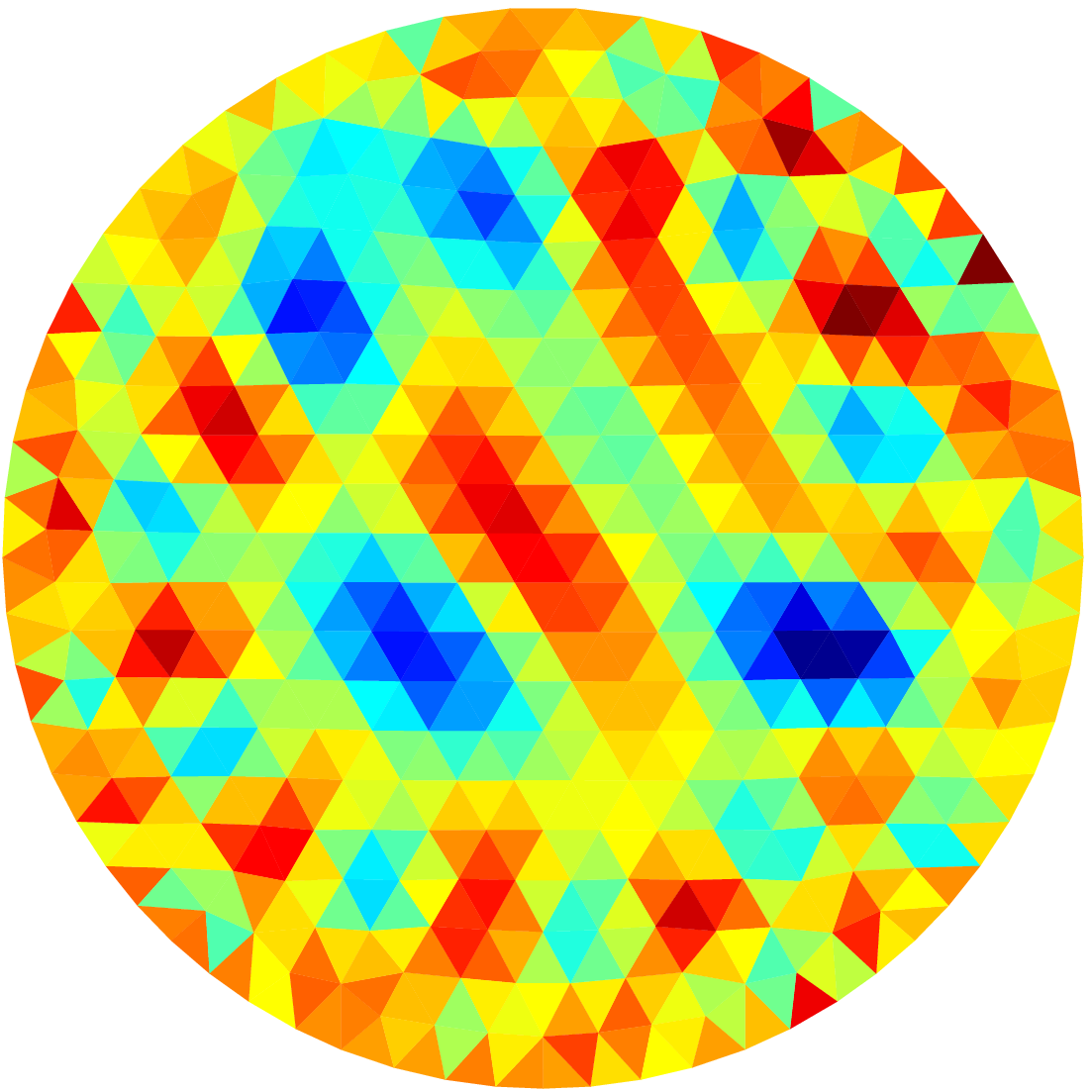}&\includegraphics[width=0.22\linewidth]{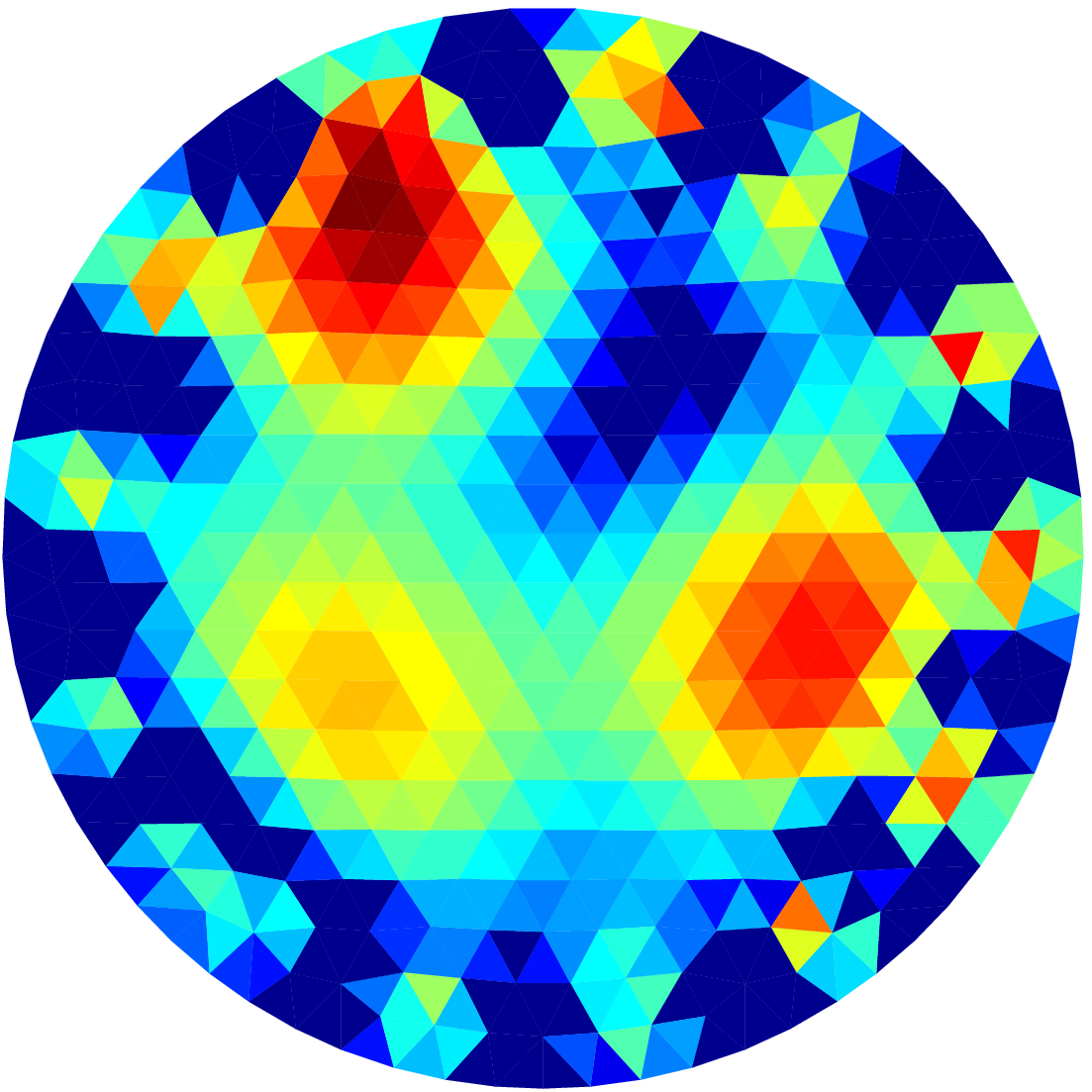}&
\includegraphics[width=0.22\linewidth]{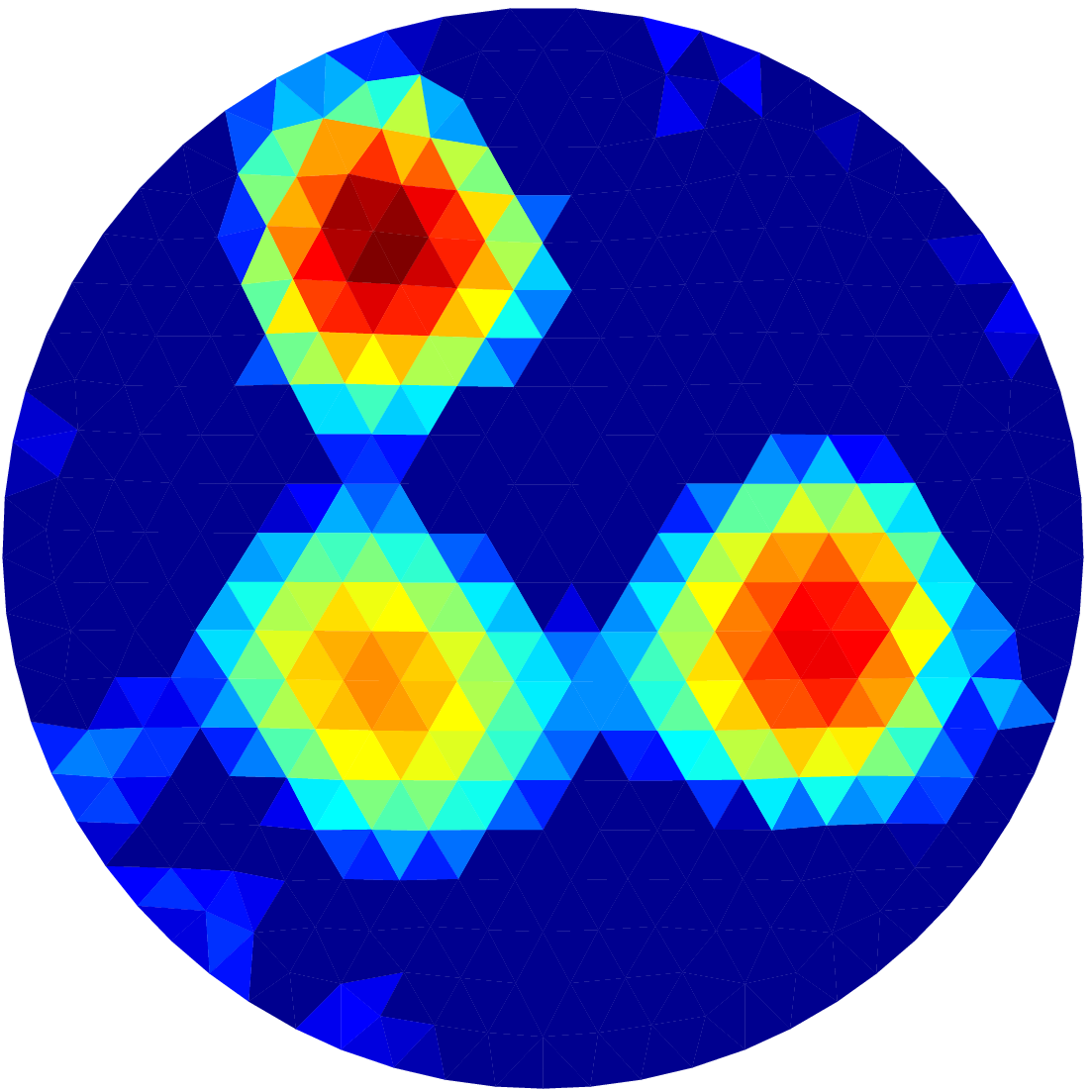}\\
 \begin{tabular}{c}\small(a) true pressure\\ \small  \end{tabular}  &  \begin{tabular}{c}\small(b) conventional\\ \small method \end{tabular} & \begin{tabular}{c}\small(c) reconstruction\\ \small algorithm($\delta< h$)\end{tabular}  &  \begin{tabular}{r}\small(d) proposed\\ \small method \end{tabular} \\
\includegraphics[width=0.22\linewidth]{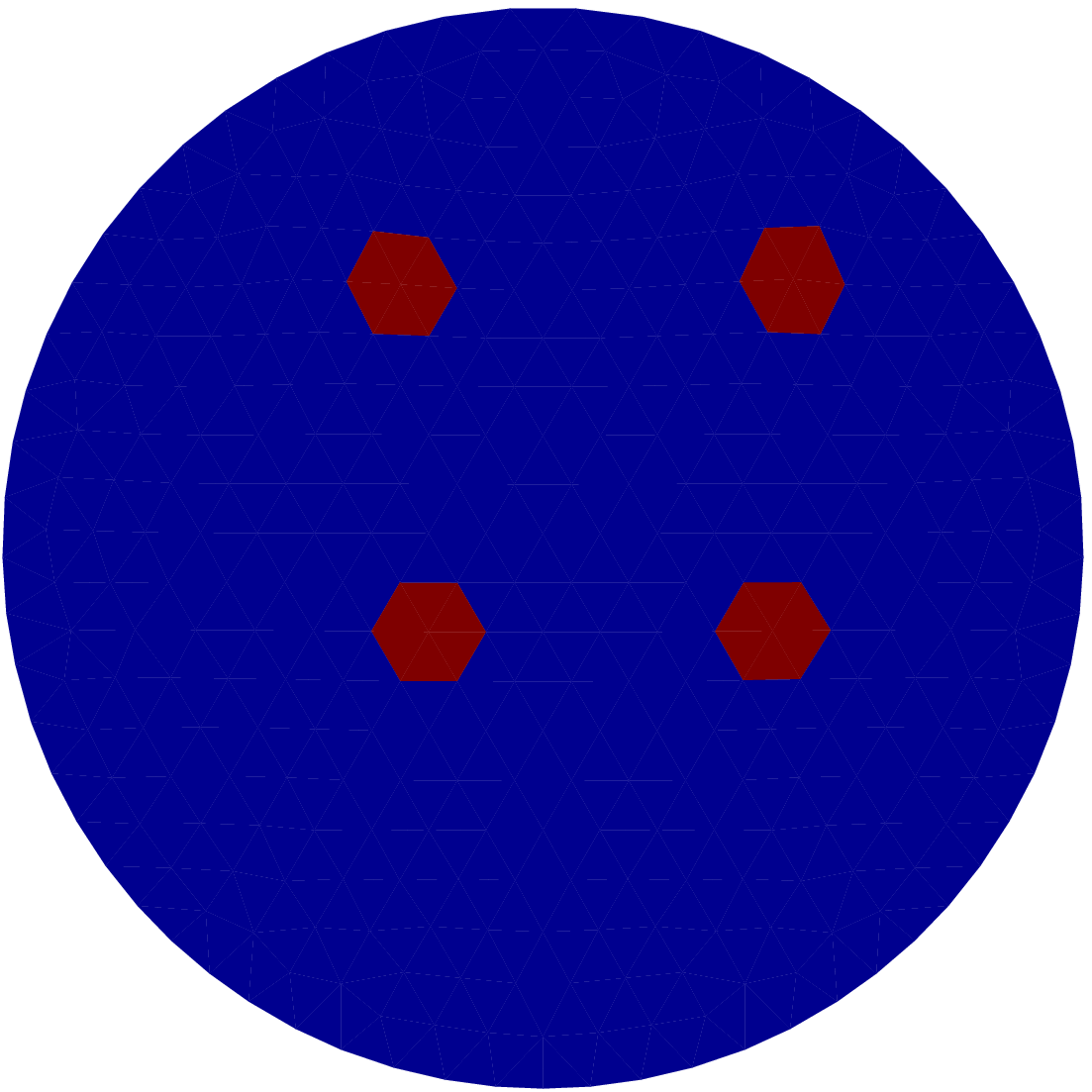}&\includegraphics[width=0.22\linewidth]{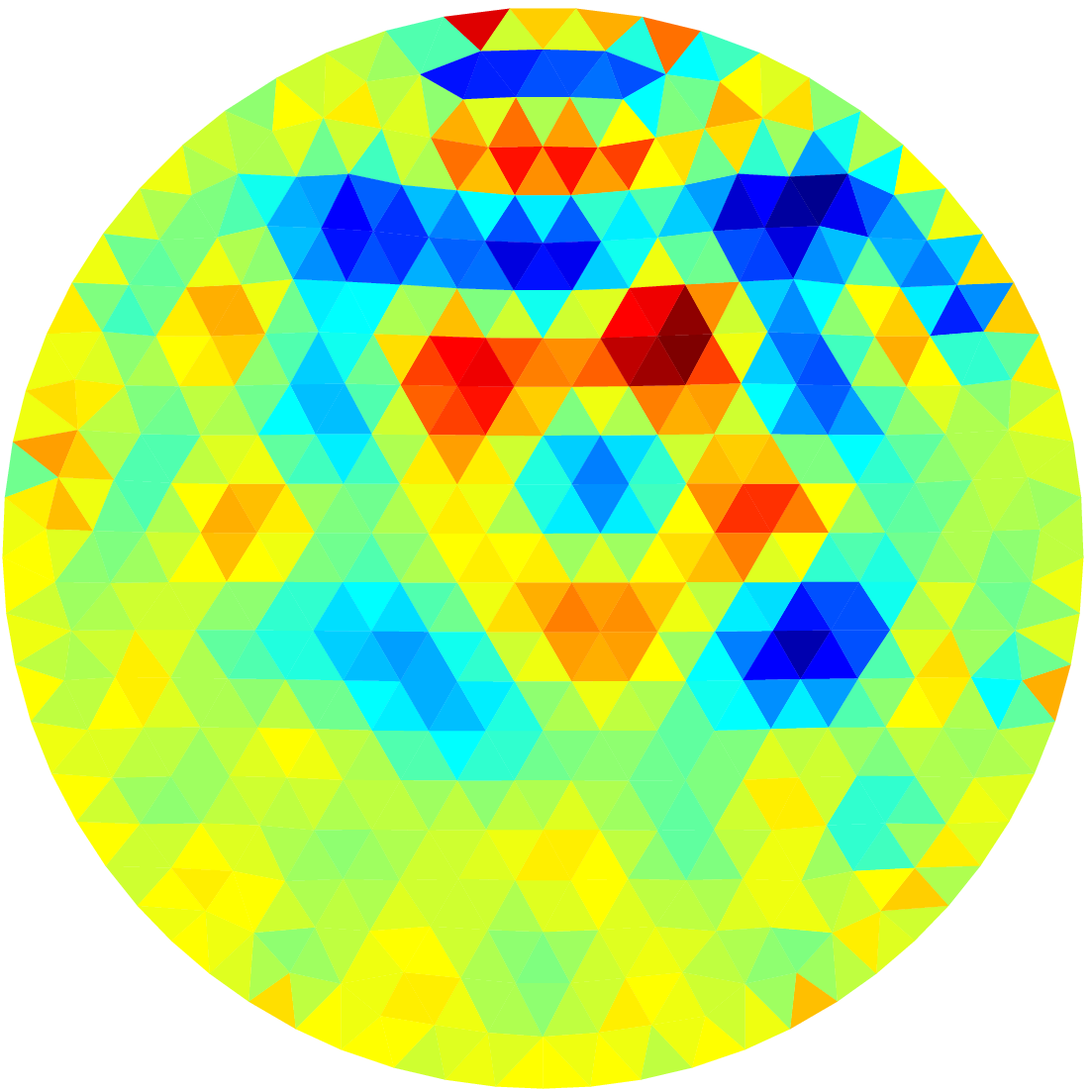}&\includegraphics[width=0.22\linewidth]{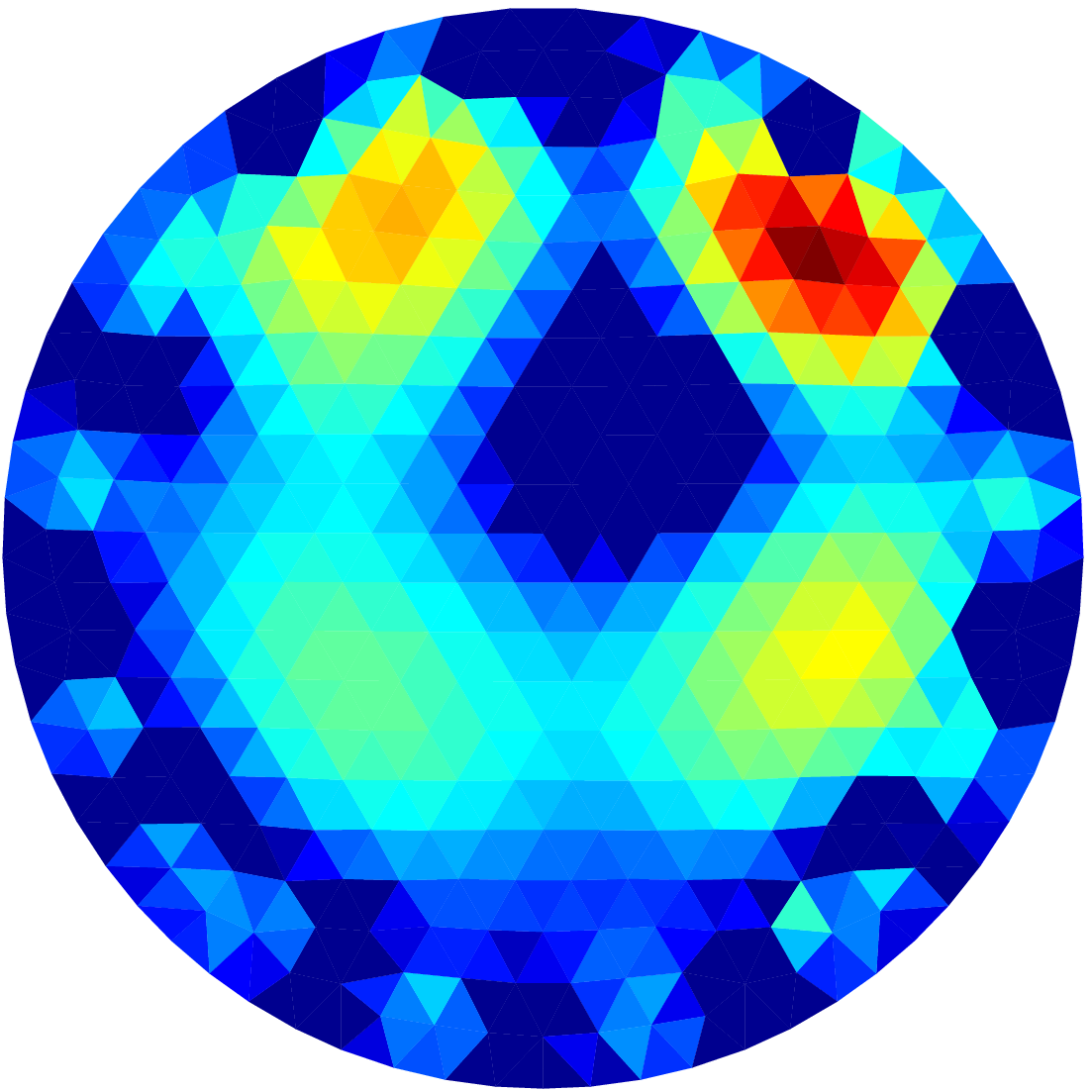}&\includegraphics[width=0.22\linewidth]{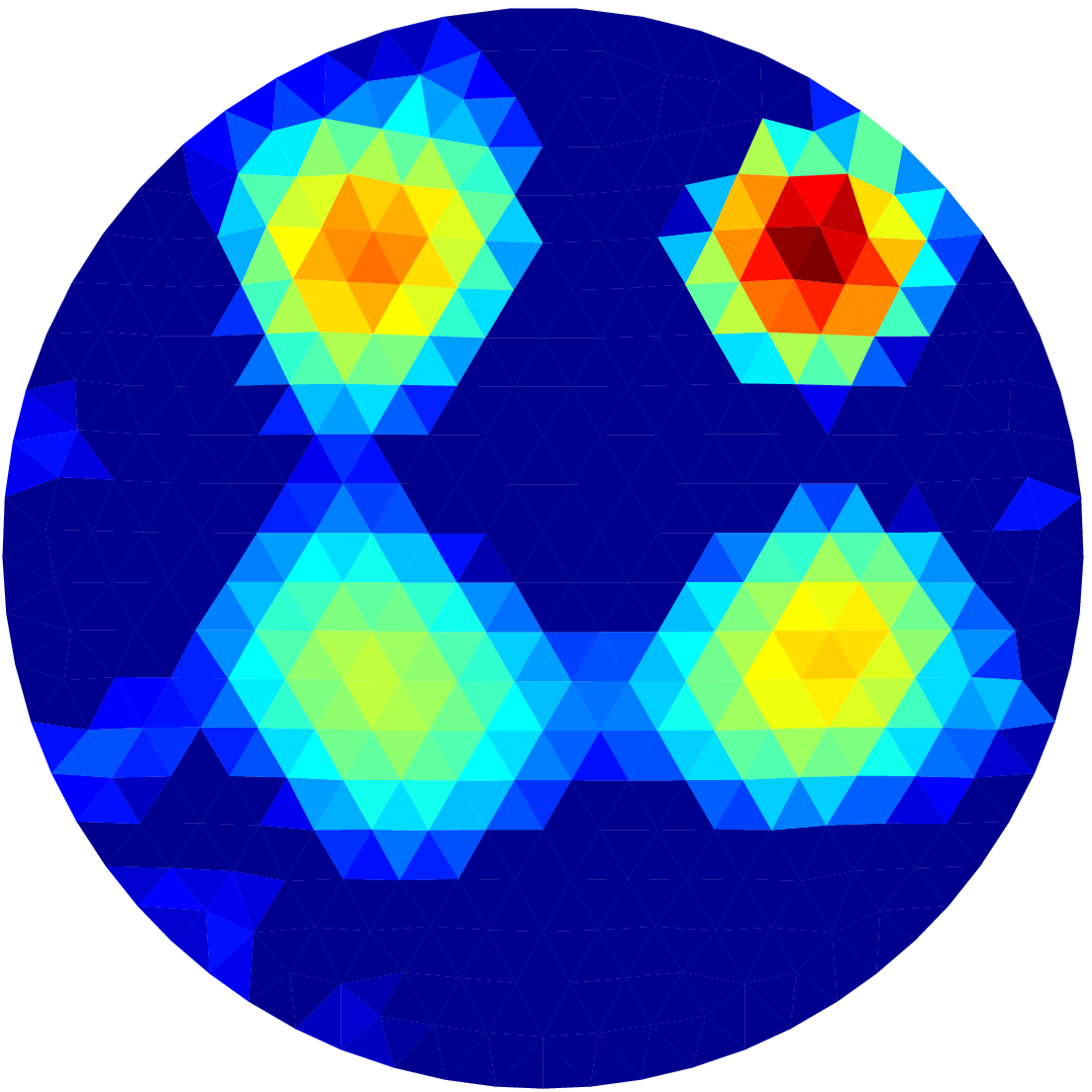}\\
 \begin{tabular}{c}\small(e) true pressure\\ \small  \end{tabular}  &  \begin{tabular}{c}\small(f) conventional\\ \small method \end{tabular} & \begin{tabular}{c}\small(g) reconstruction\\ \small algorithm($\delta< h$)\end{tabular}  &  \begin{tabular}{r}\small(h) proposed\\ \small method \end{tabular} \\
 \end{tabular}
\caption{\label{simu2}True pressure distributions(first column), the conductivity variation images by conventional EIT-method(second column), and the reconstructed pressure images using the reconstruction algorithm \ref{recon_algo} with $\delta< h$(third column), $\delta=5h$(fourth column). Here, $h$ is the side length of the triangular mesh.}
\end{figure}

\section{Conclusion}

We have provided a mathematical framework of an EIT-based pressure-sensor for the development of an image reconstruction algorithm.  We have derived the first mathematical model describing the electromechanical properties of a conductive membrane with the standard EIT system.  We have found that the geometric variation of the membrane due to an applied pressure produces anisotropic conductance variation. Hence, the corresponding inverse problem of recovering the anisotropic conductivity distribution cannot be addressed  by the EIT method alone. Under the assumption that the geometric variation is not too large, we have developed a reconstruction algorithm based on a sensitivity matrix of  current-voltage data arising from small perturbations of pressure. Numerical simulations have indicated the  feasibility of the EIT-based pressure sensor by successfully recovering pressures.

The proposed model is based on the incompressibility assumption which may not be satisfied by various flexible conductive materials, such as conductive fabrics. Recently,
Bera \etal   \cite{Bera2014} investigated the electromechanical properties of different conductive fabrics through  electrical impedance spectroscopy.  Because fabrics are flexible, durable, and washable, EIT-based fabric pressure sensors should be widely applicable as, for example, wearable sensors.   Constructing a suitable mathematical model will be very complicated because the interactions of the conductive yarns and the air gaps among them should be taken into account. This will be the subject of future studies.


\end{document}